\documentclass[10pt,a4paper]{article}      

\usepackage{mathtools}

\mathtoolsset{showonlyrefs}

\usepackage{mathrsfs}

\usepackage{amsfonts, amsmath}

\usepackage{amssymb,amsthm}

\usepackage{
	latexsym,
	nicefrac,
}

\usepackage[usenames]{color}

\usepackage{url}

\definecolor{darkgreen}{rgb}{0,0.5,0}
\definecolor{darkred}{rgb}{0.7,0,0}
\usepackage[colorlinks, 
citecolor=darkgreen, linkcolor=darkred
]{hyperref}

\theoremstyle{plain}
\newtheorem{lemma}{Lemma}[section]
\newtheorem{thm}[lemma]{Theorem}
\newtheorem{prop}[lemma]{Proposition}
\newtheorem{cor}[lemma]{Corollary}

\theoremstyle{definition}
\newtheorem{defn}[lemma]{Definition}

\newtheorem{rmk}[lemma]{Remark}

\newcommand{\N}{\ensuremath{{\mathbb N}}}

\newcommand{\tensor}{\otimes}
\newcommand{\lap}{\Delta}

\DeclareMathOperator{\inj}{inj}

\DeclareMathOperator{\Sym}{Sym}

\newcommand{\norm}[1]{\left \Vert#1 \right \Vert}  

\newcommand{\Rm}{{\mathrm{Rm}}}

\newcommand*\ddt{\frac{\mathrm{d}}{\mathrm{d}t}}
\newcommand*\pddt{\frac{\partial}{\partial t}}

\newcommand{\pt}{\partial_t}

\newcommand{\eps}{\varepsilon}

\newcommand*\tr{\mathop{\mathrm{tr}}\nolimits}
\newcommand{\Tau}{{\cal T}}

\usepackage{parskip}

\title{Limiting Behaviour of the Teichm\"uller Harmonic Map Flow}
\author{Tobias Huxol}
\date{\today}
\begin{document}
	\maketitle
\begin{abstract}
	In this paper we study the Teichm\"uller harmonic map flow as introduced by Rupflin and Topping \cite{RT}. It evolves pairs of maps and metrics $(u,g)$ into  branched minimal immersions, or equivalently into weakly conformal harmonic maps, where $u$ maps from a fixed closed surface $M$ with metric $g$ to a general target manifold $N$. It arises naturally as a gradient flow for the Dirichlet energy functional viewed as acting on equivalence classes of such pairs, obtained from the invariance under diffeomorphisms and conformal changes of the domain metric. 
	
	In the construction of a suitable inner product for the gradient flow a choice of relative weight of the map tangent directions and metric tangent directions is made, which manifests itself in the appearance of a coupling constant $\eta$ in the flow equations.
	We study limits of the flow as $\eta$ approaches 0, corresponding to slowing down the evolution of the metric.
	
	We first show that given a smooth harmonic map flow on a fixed time interval, the Teichm\"uller harmonic map flows starting at the same initial data converge uniformly to the underlying harmonic map flow when $\eta \downarrow 0$.
	
	Next we consider a rescaling of time, which increases the speed of the map evolution while evolving the metric at a constant rate. We show that under appropriate topological assumptions, in the limit the rescaled flows converge to a unique flow through harmonic maps with the metric evolving in the direction of the real part of the Hopf differential.
\end{abstract}
\section{Introduction}
Given a smooth closed orientable surface $M$, a metric $g$ on $M$ and a smooth closed Riemannian manifold $N=(N,G)$, the \emph{energy} of a map $u: M \to N$ is defined as
\begin{equation}
E(u,g) = \frac{1}{2} \int_M |du|^2_g dv_g.
\end{equation}
We call $u$ a harmonic map if it is a critical point of $E$ (viewed as a functional on maps). The \emph{harmonic map flow} is then given by
\begin{equation}
\label{eq:hmf}
\frac{ \partial u}{\partial t} = \tau_g(u),
\end{equation}
where $\tau_g(u) = \tr \nabla du$ is the tension field of the map $u$. One can view this flow as a gradient flow (with respect to the map) for the energy functional $E$.
The harmonic map flow has been studied extensively, see e.g. \cite{CDY,ES, hartman,parker, struweCMH}. As expected by its gradient flow nature, it aims to transform an initial map into a harmonic map. In general, this is not possible without singularities in the map forming, as some homotopy classes do not contain \emph{any} harmonic maps (\cite{EL_2R}). However, assuming nonpositive sectional curvature on $N$, Eells and Sampson were able to show long-time existence and uniform derivative bounds for this flow in \cite{EL} and as a result could show that any homotopy class of maps contains a harmonic representative.

If one also allows the domain metric to vary, a different gradient flow for the functional $E$ can be found. The energy is invariant under conformal changes of the domain metric in two dimensions, so depending on the genus of $M$ one can restrict to flowing through metrics of Gauss curvature $K \in \{-1,0,1\}$, corresponding to hyperbolic surfaces, tori and the sphere respectively. This gradient flow was introduced in \cite{RT}, and is called the \emph{Teichm\"uller harmonic map flow}. It is given by
\begin{align}
\begin{split}
\label{eq:flow}
\pddt u &= \tau_g (u) \\
\pddt g &= \frac{\eta^2}{4} Re(P_g(\Phi(u,g)))
\end{split}
\end{align}
where $\tau_g$ again denotes the tension field of $u$, the coupling constant $\eta$ a choice of scaling in defining a metric on pairs of maps and metrics, $P_g$ the $L^2$-orthogonal projection of quadratic differentials onto holomorphic quadratic differentials and $\Phi(u,g)$ is the Hopf differential of $u$, a quadratic differential that measures how `close' $u$ is to being conformal. In particular vanishing of $\Phi(u,g)$ implies that $u$ is a weakly conformal map. This flow thus tries to transform given initial data into a weakly conformal harmonic map.

A result of Gulliver, Osserman and Royden (\cite{GOR}) allows us to give a more geometric characterization of such maps as branched minimal immersions. 

In this sense one can view the study of the flow \eqref{eq:flow} as an alternative approach to the classical variational method for studying branched minimal immersions (as done in \cite{SU}), and indeed in \cite{RT} some of the results in \cite{SU} are recovered using the flow method.

A priori singularities in this flow might appear both in the metric (i.e. parts of the domain may become arbitrarily `thin') and the map. The latter type is well-understood for the harmonic map flow, the general idea is that so-called bubbles can be extracted at points and times where energy concentrates (\cite{struweCMH}). These bubbles are non-constant harmonic maps $S^2 \to N$. Away from such concentration points, the harmonic map flow enjoys higher regularity. This principle carries over to the Teichm\"uller harmonic map flow (\cite{rupflin_existence}), and allows one to employ a lot of the techniques familiar from the study of the harmonic map flow in the development of the theory for the flow \eqref{eq:flow}.

Possible degeneration of the metric on the other hand requires new ideas. When this degeneration only happens at infinite time and $M$ is a hyperbolic surface, it was shown in \cite{RTZ} that the initial map in some sense (sub-)converges to a collection of branched minimal immersions (or constant maps) in the limit. In a joint work with Rupflin, Topping and the author (\cite{HRT}) the properties of this convergence where studied more closely.

When one constructs the flow \eqref{eq:flow} as a formal gradient flow on a space of maps and hyperbolic metrics (see \cite{RT} for the details of this construction), a choice of inner product is made. This manifests itself in the coupling constant $\eta$ in the definition of the flow \eqref{eq:flow}. In this paper we study the behaviour of the flow when one lets $\eta \downarrow 0$. 

We establish two results in this direction. Initially, we fix some initial data $(u_0,g_0)$ and a time interval $[0,T]$. Assuming that the corresponding classical harmonic map flow starting from the same initial data does not develop singularities, we prove the following theorem, establishing uniform convergence of the Teichm\"uller harmonic map flow to the harmonic map flow:
\begin{thm}
	\label{thm:eta0}
	Let $M$ be a smooth closed oriented surface of genus $\gamma \geq 2$ and $g_0 \in {\cal{M}}_{-1}$, the space of constant Gauss curvature $-1$ hyperbolic metrics on $M$. Further take $(N,G)$ to be a  smooth closed Riemannian manifold and $u_0:M \to N$ a smooth map.  Consider a fixed time interval $[0,T]$ such that the harmonic map flow $u(t)$ satisfying \eqref{eq:hmf} (with respect to $g_0$), starting at the  initial condition $u(0) = u_0$, is smooth. Then the flows $( u_\eta(t), g_\eta(t))$ satisfying \eqref{eq:flow} with initial condition $(u_0,g_0)$ and some coupling constant $\eta > 0$ converge smoothly to the harmonic map flow $u(t)$, in the following sense as $\eta \downarrow 0$:
	\begin{enumerate}
		\item \label{test} The metrics $g_\eta(t)$ converge to the initial metric $g_0$ in $C^k(M,g_0)$ uniformly in $t$ for each $k \in \mathbb{N}$.
		\item The maps $u_\eta(t)$ converge to $u(t)$ smoothly on $M \times [0,T]$.
	\end{enumerate}
\end{thm}
The proof will be carried out in Section \ref{sec:etasmall}.

We can also adopt a different viewpoint on the relative behaviour of the map evolution and the metric evolution for small $\eta$. By rescaling  time appropriately, we can think of the map evolution happening  increasingly quickly, while we fix the `speed' of the metric evolution. In particular the rescaling $t \to \frac{4}{\eta^2}t$ yields

\begin{align}
\begin{split}
\label{eq:flow2}
\pddt u &= \kappa \tau_g (u) \\
\pddt g &= Re(P_g(\Phi(u,g))),
\end{split}
\end{align}
where $\kappa := \frac{4}{\eta^2}$.
Studying $\eta \downarrow 0$ now corresponds to $\kappa \to \infty$. In the limit, we heuristically expect the map to become instantaneously harmonic at all times, while the metric evolves in the direction of the real part of the Hopf differential (as $Re(P_g(\Phi(u,g))) = Re(\Phi(u,g))$ for $u$ harmonic). 

Consider a target $(N,G)$ with strictly non-positive curvature and fix a homotopy class of maps which does not contain any constant maps or maps to closed geodesics. Work by Hartman (\cite{hartman}) then guarantees the existence of a unique harmonic map in that homotopy class for any choice of metric on the domain $M$. Thus given a curve of metrics $g(t):[0,T] \to \mathcal{M}_{-1}$, we can find a corresponding curve of harmonic maps $u(t): M \times [0,T] \to N$.
In this setting we prove that the flows \eqref{eq:flow2} do indeed converge to a flow through harmonic maps.

\begin{thm}
	\label{thm:rescaled}
	Let $M$ be a smooth closed oriented surface of genus $\gamma \geq 2$ and $(N,G)$ be a smooth closed Riemannian manifold.
	Given smooth initial data $(u_0,g_0)$ for  \eqref{eq:flow2}, take $0 < T \leq T_0$ with $T_0$ from Lemma \ref{lem:injectivity}, and consider the sequence $(u_\kappa(t),g_\kappa(t))_{\kappa=1}^{\infty}$ of solutions to \eqref{eq:flow2} with rescaled coupling constant $\kappa$ on the fixed time interval $[0,T]$, which we further assume to be smooth up to $t=T$. Then the following is true:
	
	\begin{enumerate}
		\item
		\label{claim:1}
		There exists a limit curve of hyperbolic metrics $g$ (i.e. each $g(t)$ has Gauss curvature $K=-1$) on $[0,T]$, continuous in time and smooth in space in the sense that for all $k \in \mathbb{N}$, $g$ is an element of $C^{0} ([0,T], C^k(\Sym^2(T^*M),g_0))$. 
		After possibly selecting a subsequence in $\kappa$,
		The curves $g_\kappa$ converge to $g$ in $C^{0} ([0,T], C^k(\Sym^2(T^*M),g_0))$ (i.e. uniformly in time in $C^k(M,g_0)$), again for all $k \in \mathbb{N}$. 
		\item 	
		\label{claim:2}
		Further assume that $N$ has strictly negative sectional curvature and that the homotopy class of $u_0$ does not contain maps to closed geodesics in the target or constant maps. Let $u(t): M \times (0,T] \to N$ be the unique curve of harmonic maps homotopic to $u_0$ corresponding to $g(t)$, then the limit curve of metrics $g$ is differentiable in time at each point $x \in M$ away from $t=0$, with derivative given by  $\ddt g(t)(x) = Re( \Phi(u,g))(x)$, where $\Phi(u,g)$ as usual denotes the Hopf differential. Finally, the maps $u_\kappa(t)$ also converge to $u(t)$ uniformly in $t$ in $C^k(M,g_0)$ away from 0 for all $k \in \mathbb{N}$, and the convergence of $(u_\kappa, g_\kappa)$ to $(u(t),g(t))$ does not require a choice of subsequence in $\kappa$.
	\end{enumerate}
\end{thm}

We give the proof and the definitions of the involved spaces in Section \ref{sec:rescaled}. 

\begin{rmk}
	
A consequence of Theorem \ref{thm:rescaled} is that the flow through harmonic maps

\begin{align}
\begin{split}
\label{eq:flow3}
\tau_g(u) &= 0 \\
\pddt g &= Re(\Phi(u,g)),
\end{split}
\end{align}

enjoys short-time existence when one works in the above setting. On the way to proving Theorem \ref{thm:rescaled} we in fact also show a uniqueness statement. It would be interesting to analyse this flow further, in particular investigating whether the resulting curves in Teichm\"uller space are geodesics (with respect to the Weil-Petersson metric) and if finite-time singularities in the metric can occur (see e.g. \cite{tromba}).
\end{rmk}

\emph{Acknowledgements:} Large parts of the work in this paper were originally carried out during my thesis, and I would like to thank my advisor Peter Topping for his guidance and encouragement. I also enjoyed the support of the Pacific Institute for the Mathematical Sciences (PIMS) during the later stages of this project.

\section{Small Coupling Constant Limit}
\label{sec:etasmall}
In this section we prove Theorem \ref{thm:eta0}. To this end, we first establish the metric convergence using estimates from \cite{rupflin_existence}. We then prove that the map part converges, which is accomplished through showing that the energy concentration along the flow \eqref{eq:flow} stays controlled if the corresponding harmonic map flow is smooth.

\subsection{Convergence of the Metric}	

We collect some results on the behaviour of the metric $g(t)$ under the flow \eqref{eq:flow} here. 

\begin{defn}
	Let $g(t)$ be a smooth one-parameter family of metrics for $t \in [0,T]$.
	Define the $L^2$-length of $g(t)$ on $[0,T]$ by
	\begin{equation*}
	L(g, [0,T]) = \int_{0}^{T} \norm{ \pt g }_{L^2(M,g(t))} dt.
	\end{equation*}
\end{defn}

Let $(u(t),g(t))$  be evolving under $\eqref{eq:flow}$. Recall the energy identity
\begin{equation}
\label{eq:energy-identity}
\frac{dE(u,g)}{dt}=-\int_M 
\left[|\tau_g(u)|^2+\left(\frac{\eta}{4}\right)^2 |Re(P_g(\Phi(u,g)))|^2\right]dv_g.
\end{equation}
from \cite{RT}. As a consequence we have an estimate for the $L^2$-length.

\begin{lemma}
	Let $T>0$, $\eta>0$ and assume that $(u_\eta,g_\eta(t))$ is a solution to $\eqref{eq:flow}$ with coupling constant $\eta$ on $[0,T]$ for some initial data $(u_0,g_0)$. Then we have the estimate
	\label{lem:l2small}
	\begin{equation}
	\label{eq:L2small}
		L(g_\eta, [0,T]) \leq \eta \sqrt{ T E(u_0,g_0)}.
	\end{equation}
\end{lemma}
\begin{proof}
	Integrating the energy identity \eqref{eq:energy-identity} in time from $0$ to $T$ yields
	\begin{equation*}
	\int_{0}^{T} \int_M 	\left(\frac{\eta}{4}\right)^2 |Re(P_{g_\eta}(\Phi(u_\eta,g_\eta)))|^2 dv_{g_\eta} dt \leq E_0.
	\end{equation*}
	for $E_0 = E(u_0,g_0)$ denoting the initial energy. Hence
	\begin{equation*}
	\int_{0}^{T} \norm{ \pt g_\eta  }^2_{L^2(M,g_\eta(t))} dt \leq \eta^2 E_0,
	\end{equation*} and H\"older's inequality yields the claim.
\end{proof}

\begin{rmk}
	Note that after projecting a curve of metrics $g(t)$ down to a path $[g(t)]$ in Teichm\"uller space the $L^2$-length of $g(t)$ defined in the above lemma corresponds to the length of $[g(t)]$, computed with respect to the classical Weil-Petersson metric (up to a constant) .
\end{rmk}

We call curves of metrics evolving in the direction of the real part of a holomorphic quadratic differential \emph{horizontal curves}: 

\begin{defn}[See also \cite{RThorizontal}]
	\label{def:horizontal}
	In our setting a \emph{horizontal curve} is a smooth one-parameter family of hyperbolic metrics $g(t)$ on $M$ for $t \in [t_1,t_2]$, such that for all such $t$ we have $\ddt g(t) =  Re( \Psi(t))$, where $\Psi(t)$ is some holomorphic quadratic differential on $(M,g(t))$.
\end{defn}

In particular solutions $g_\eta(t)$ to \eqref{eq:flow} are horizontal curves.
For such horizontal curves we state an estimate from \cite{rupflin_existence} that applies when their $L^2$-length is small.

\begin{prop}[Proposition 2.2 in \cite{rupflin_existence}]
	\label{prop:metricbounds}
	For every $\epsilon > 0$ and every $s > 3$ there exists a number $\theta = \theta(\epsilon, s) > 0$ such that the following holds true. Let $g_0 \in \mathcal{M}^s_{-1}$, i.e. a hyperbolic metric of class $H^s$, for which the length $\ell(g_0)$ of the shortest closed geodesic in $(M, g_0)$ is no less than $\epsilon$. Then there is a number $C = C(g_0, s) < \infty$ such that for any horizontal curve $g(t)$ with $g(0) = g_0$ and $L(g,[0,T]) \leq \theta$
	we have
	\begin{equation} 
	\label{eq:hsl2}
	\norm{ \ddt g(t) }_{H^s} \leq C \norm{ \ddt g(t) }_{L^2(M,g(t))} \, \text{for every} \,\, t \in [0,T).
	\end{equation}
\end{prop}

The $H^s$-norm here and in the following is to be understood with respect to some fixed set of local coordinate charts on $M$, thus in particular $\mathcal{M}^s_{-1}$ refers to the set of hyperbolic metrics with coefficients in $H^s(M)$.

We also have the following result from \cite{rupflin_existence} controlling the projection $P_g$.

\begin{lemma}[Lemma 2.9 in \cite{rupflin_existence}]
	\label{lem:uniformPg}
	For any $g_0 \in \mathcal{M}_{-1}$ and any $s > 3$ there exists a neighbourhood $W$ of $g_0$ in $\mathcal{M}^s_{-1}$ and a constant $C=C(g_0,s)$ such that for all $g \in W$ and $h \in \Gamma(\Sym^2(T^*M))$ we have
	\begin{equation}
	\norm{P_g{h}}_{H^s} \leq C\norm{h}_{L^1(M,g_0)}.
	\end{equation}
\end{lemma}

As a consequence of Lemma \ref{lem:l2small}	we can apply Proposition \ref{prop:metricbounds} for all sufficiently small $\eta$. We can then integrate \eqref{eq:hsl2}, and estimating as in \eqref{eq:L2small} we find
\begin{equation}
\label{eq:Ckbound}
\norm{ g_\eta(t) - g_0}_{C^k(M,g_0)} \leq C \eta \sqrt{t} \leq C \eta.
\end{equation}

Here $C$ depends on $T$, $k$, $g_0$ and $E_0 = E(u_0,g_0)$ and we used the fact that $H^s$ embeds continuously into $C^k$ for sufficiently large $s$. We first collect here a number of useful consequences of the above results.

\begin{rmk}
	\label{rmk:eta}
	
	By the previous estimates we can find $\eta_0 = \eta_0 (T,g_0, E_0, N) \leq 1$ such that for $\eta \leq \eta_0$ and $t \in [0,T]$ the following properties hold:
	\begin{enumerate}
		\item \label{Claim:eta1} For any vector field $X \in \Gamma(TM)$ we have $\frac{1}{2} |X|_{g_0} \leq |X|_{g_\eta(t)} \leq 2 |X|_{g_0}$.
		\item \label{Claim:eta2} For any smooth map $u: M \to N$ we have $\frac{1}{2} |du|_{g_0} \leq |du|_{g_\eta(t)} \leq 2 |du|_{g_0}$.
		\item \label{Claim:eta3} For any $(0,2)$-tensor $h \in \Gamma( \Sym^2 T^*M )$ we have $\frac{1}{2} |h|_{g_0} \leq |h|_{g_\eta(t)} \leq 2 |h|_{g_0}$.
		\item \label{Claim:eta5} The injectivity radius satisfies $\frac{1}{2} \inj_{g_0} \leq \inj_{g_\eta(t)} \leq 2 \inj_{g_0}$. In particular we necessarily have existence of a weak solution $(u_\eta(t),g_\eta(t))$ in the sense of \cite{rupflin_existence} up to time $T$.
		\item \label{Claim:eta6} For all $x \in M$ and $r > 0$ metric balls satisfy $B^{g_0}_{\frac{r}{2}}(x) \subset B^{g_{\eta}(t)}_{r}(x) \subset B^{g_0}_{2r}(x)$ .
		\item The difference of the inverse metric tensors is bounded by $|g_0^{ij}-g_\eta(t)^{ij}| \leq C\norm{g_0-g_\eta(t)}_{C^0(M,g_0)}$. Similarly, the difference of the Christoffel symbols of $g_0, g_\eta(t)$ is bounded by $C\norm{g_0-g_\eta(t)}_{C^1(M,g_0)}$.
		
		\item \label{Claim:eta} The metrics $g_\eta(t)$ lie in the neighbourhood $W$ from Lemma \ref{lem:uniformPg}, and in particular 
		$\norm{ \ddt g_\eta(t) }_{C^k(M,g_0)} \leq C \eta^2 \leq C$ with $C$ depending only on $k$, $g_0$ and $E_0$.

	\end{enumerate}

\end{rmk}

An immediate consequence is the desired metric convergence on fixed time intervals.

\begin{cor}
	\label{cor:metriceta0}
	Under the assumptions of Theorem \ref{thm:eta0}, we have uniform convergence in $t$ of $g_\eta(t)$ to $g_0$ on $[0,T]$ in $C^k(M,g_0)$ for any $k \in \mathbb{N}$ as $\eta \downarrow 0$.
\end{cor}

\begin{proof}
The convergence can be seen as in \eqref{eq:Ckbound}. Note that the above remark guarantees that the flows considered actually exist up to time $T$ (as degeneration of the metric is ruled out for sufficiently small $\eta$).
\end{proof}

\subsection{$L^2$-Convergence of the Map for small times}
We proceed by proving  $L^2$-closeness of the flow \eqref{eq:flow} for small $\eta$ to the underlying harmonic map flow under a boundedness assumption we will justify later. 

\begin{prop}

	\label{prop:l2flowsmall}
	Take $M$ to be a smooth closed oriented surface of genus $\gamma \geq 2$ and $N$ to be a smooth closed Riemannian manifold. As before denote by $u(t)$, $u_\eta(t)$  the solutions (with respect to initial data $(u_0,g_0)$) to the usual harmonic map flow \eqref{eq:hmf}, respectively Teichm\"uller harmonic map flow \eqref{eq:flow} with coupling constant $\eta$. Let $T>0$ be such that the flow $u(t)$ is smooth on $[0,T]$. Further assume there exist constants $C_1 > 0, \eta_1 > 0$ (allowed to depend on $M,N, u_0,g_0$ and $T$) such that the flow $u_\eta(t)$ admits a smooth solution for all $\eta \leq \eta_1$ which satisfies
	\begin{equation}
	\int_{0}^{T} \int_M |\nabla u_\eta(t)|^4_{g_0} dv_{g_0} dt \leq C_1.
	\end{equation} 
	
	Then given any $\eps > 0$, we can find $\eta_\eps > 0$, depending on $\eps$, $\eta_1$, $C_1$, $T$, $u_0$, $(M,g_0)$ and $N$, such that for all $\eta_\eps \geq \eta > 0$ we have  
	\begin{equation}
	\norm{u_\eta(t) - u(t)}_{L^2(M,g_0)} < \eps
	\end{equation}
	for $t \in [0,T]$.
	
\end{prop}
\begin{proof}
		We use techniques from \cite[Section 3]{rupflin_existence}, based on \cite{struweCMH}.  In particular the difference $w = u - u_\eta$ satisfies the evolution equation
		\begin{equation}
		\pddt w - \lap_{g_0} (w) = (\lap_{g_0} - \lap_{g_\eta}) (u_\eta) + A_{g_0} (u) ( \nabla u, \nabla u) - A_{g_\eta} (u_\eta) (\nabla u_\eta, \nabla u_\eta),
		\end{equation}
		where $A$ denotes the second fundamental form of the target $N \hookrightarrow \mathbb{R}^n$ and we write $A_g(u)(\nabla u, \nabla u) = g^{ij}A(u)(\partial_i u, \partial_j u)$.
		
		Multiplying this equation by $w$ and integrating over $(M,g_0)$ together with partial integration (with respect to $g_0$) we obtain
		\begin{align}
		\label{eq:l2small0}
		\begin{split}
		\frac{1}{2} \ddt \norm{w}^2_{L^2(M,g_0)} + \norm{\nabla w}^2_{L^2(M,g_0)}  = &\int_M (\lap_{g_0} - \lap_{g_\eta}) (u_\eta)w dv_{g_0}  \\ + &\int_M A_{g_0} (u) ( \nabla u, \nabla u)w dv_{g_0} \\
		- &\int_M A_{g_\eta} (u_\eta) (\nabla u_\eta, \nabla u_\eta)w dv_{g_0}.
		\end{split}
		\end{align}
		We now estimate the terms on the right hand side. All norms in the following, as well as the volume form $dv$, are taken to be defined with respect to $g_0$. Also assume from now on that $\eta \leq \eta_0$ so Remark \ref{rmk:eta} applies, then for the first term we find
		\begin{align} \label{eq:l2small1}
		\left|\int_M (\lap_{g_0} - \lap_{g_\eta})(u_\eta)wdv \right| \leq C  \norm{g_0-g_\eta}_{C^0} \norm{\nabla w}_{L^2} \norm{\nabla u}_{L^2}  \leq C  \norm{g_0-g_\eta}_{C^0}  \norm{\nabla w}_{L^2}
		\end{align}
		where we used integration by parts (with respect to both $g_0$ and $g_\eta$, as the flow \eqref{eq:flow} leaves the volume form invariant, see \cite{RT}), and Remark \ref{rmk:eta} to estimate the difference of the inverse metric tensors and finally $\norm{\nabla w}^2_{L^2} \leq C E_0 \leq C$.
		
		Here and in the following we let $C$ denote positive constants, which may depend on $(M,g_0)$, $N$ and the initial energy $E(u_0,g_0)$.
		
		We proceed to estimate the second fundamental form terms. Note that
		\begin{align}
		\begin{split}
		\label{eq:A1}
		A_{g_0} (u) ( \nabla u, \nabla u) - A_{g_\eta} (u_\eta) (\nabla u_\eta, \nabla u_\eta) =& A_{g_0} (u) ( \nabla u, \nabla u) -  A_{g_0} (u_\eta) (\nabla u_\eta, \nabla u_\eta)\\
		+& A_{g_0} (u_\eta) (\nabla u_\eta, \nabla u_\eta) - A_{g_\eta} (u_\eta) (\nabla u_\eta, \nabla u_\eta),
		\end{split}
		\end{align}																						
		which we estimate pointwise
		\begin{align}
		\label{eq:A2}
		|A_{g_0} (u) ( \nabla u, \nabla u) -  A_{g_0} (u_\eta) (\nabla u_\eta, \nabla u_\eta)| &\leq C ( |\nabla w| |\nabla V| + |w| |\nabla V|^2) \\
				\label{eq:A3}
		|A_{g_0} (u_\eta) (\nabla u_\eta, \nabla u_\eta) - A_{g_\eta} (u_\eta) (\nabla u_\eta, \nabla u_\eta)| &\leq C \norm{g_0-g_\eta}_{C^0} |\nabla V|^2,
		\end{align}
		with $|\nabla V| := \max\{ |\nabla u|_{g_0}, |\nabla u_\eta|_{g_0}\}$. 
		The first inequality follows by an application of the mean value theorem to $A$ (which is a smooth function on $N$) (see also \cite{struweCMH} for this estimate).  The second inequality again follows by estimating the difference of the inverse metric tensors as in Remark \ref{rmk:eta}.
		Multiplying \eqref{eq:A2} and \eqref{eq:A3} by $w$, integrating over $M$ and applying H\"older's inequality gives
		\begin{align}
		\begin{split}
		\label{eq:l2small2}
		\int_M& A_{g_0} (u) ( \nabla u, \nabla u)w - A_{g_\eta} (u_\eta) (\nabla u_\eta, \nabla u_\eta)w dv
		\\ &\leq C \norm{g_0-g_\eta}_{C^0} \norm{\nabla V}^2_{L^4}\norm{w}_{L^2} + C \norm{ \nabla w }_{L^2}  \norm{\nabla V}_{L^4}  \norm{ w }_{L^4} + C \norm{\nabla V}^2_{L^4} \norm{w}^2_{L^4}.
		\end{split}
		\end{align}
		Combining \eqref{eq:l2small1} and \eqref{eq:l2small2} with \eqref{eq:l2small0} we arrive at our main estimate (see also \cite[Section 4]{rupflin_existence} for a similar estimate)
		\begin{align}
		\begin{split}
		\label{eq:intest}
		\frac{1}{2} \ddt \norm{w}^2_{L^2} + \norm{\nabla w}^2_{L^2} \leq C  \norm{g_0-g_\eta}_{C^0}  \left( 
		\norm{\nabla w}_{L^2} + \norm{w}_{L^2}   \norm{\nabla V}^2_{L^4} \right) \\ + C \norm{ \nabla w }_{L^2}  \norm{\nabla V}_{L^4}  \norm{ w }_{L^4} + C \norm{\nabla V}^2_{L^4} \norm{w}^2_{L^4}.
		\end{split}
		\end{align}
		
		The strategy is now to derive an estimate for $\ddt \norm{w(t)}_{L^2}$ that allows us to apply Gronwall's lemma to deduce our desired smallness. In particular we need to control all the terms in \eqref{eq:intest} through quantities integrable in time (e.g. $\norm{\nabla V}^4_{L^4}$ by the assumption on $u_\eta$ and the smoothness of $u(t)$),  $\norm{\nabla w}^2_{L^2}$ and $\norm{w(t)}^2_{L^2}$.
		
		To this end, recall the following consequence of Sobolev's inequality :
		\begin{equation}
		\label{eq:sobolev}
		\norm{w}^2_{L^4} \leq C \norm{w}_{L^2} ( \norm{w}_{L^2} + \norm{ \nabla w }_{L^2} ).
		\end{equation}
		
		Using Young's inequality we further find
		\begin{equation*}
		C \norm{ \nabla w }_{L^2}  \norm{\nabla V}_{L^4}  \norm{ w }_{L^4} + C \norm{\nabla V}^2_{L^4} \norm{w}^2_{L^4}
		\leq \frac{1}{4} \norm{ \nabla w}^2_{L^2} + C \norm{\nabla V}^2_{L^4} \norm{w}^2_{L^4}.
		\end{equation*}
		Together with \eqref{eq:sobolev} and Young's inequality this implies
		\begin{equation}
		C \norm{ \nabla w }_{L^2}  \norm{\nabla V}_{L^4}  \norm{ w }_{L^4} + C \norm{\nabla V}^2_{L^4} \norm{w}^2_{L^4}
		\leq \frac{1}{2} \norm{ \nabla w}^2_{L^2} 
		+ C ( 1 + \norm{\nabla V}^4_{L^4}) \norm{ w} ^2_{L^2},
		\end{equation}
		which is of the desired form.
		
		Similarly we can estimate the first term in the right hand side of \eqref{eq:intest} via Young's inequality to obtain
		\begin{multline}
		C \, \norm{g_0-g_\eta}_{C^0}  \left( \norm{\nabla w}_{L^2} + \norm{w}_{L^2}  \norm{\nabla V}^2_{L^4} \right) \\ \leq C \, \norm{g_0-g_\eta}_{C^0} ^ 2 + \frac{1}{4} \norm{\nabla w}^2_{L^2} + C \, \norm{w}^2_{L^2}   \norm{\nabla V}^4_{L^4}.
		\end{multline}
		
		We now additionally assume $\eta \leq \eta_1$, so $\norm{\nabla V}^4_{L^4} \in L^1([0,T])$. Writing  $\psi(t) = 1 + \norm{\nabla V}^4_{L^4}  \, \in L^1([0,T])$ (cf.  \cite{rupflin_existence}) we arrive at
		\begin{equation}
		\frac{1}{2} \ddt \norm{w}^2_{L^2}  \leq C \, \norm{g_0-g_\eta}_{C^0} ^ 2 + C \, \norm{w}^2_{L^2} \psi(t).
		\end{equation}
		We can apply Gronwall's lemma to this inequality for the function $\norm{w(t)}^2_{L^2}$ and finally get
		\begin{equation}
		\norm{w(t)}^2_{L^2} \leq C \int_{0}^{t} \norm{g_0-g_\eta}_{C^0}^2 e ^ {C \int_{0}^{t} \psi(s) ds}
		\\ \leq  C \int_{0}^{t} \norm{g_0-g_\eta}_{C^0}^2 e ^ {C (t + 1) }.
		\end{equation}
		
		This estimate will be valid as long as $\eta \leq \min(\eta_1,\eta_0)$.
		Hence from \eqref{eq:Ckbound} we see that we can indeed choose $\eta_\eps$ small such that 
		\begin{equation}
		\norm{u(t) - u_\eta(t)}_{L^2(M,g_0)} < \eps
		\end{equation} 
		for all $t \in [0,T]$ and $\eta \leq \eta_\eps$.
\end{proof}

The basic ingredient  for proving the $W^{1,4}$ bound required in the above proposition is following consequence of a standard small energy regularity estimate for almost-harmonic maps:

\begin{lemma}
\label{lem:h2bound}
Take $N$ to be a smooth closed manifold and $M$ a smooth closed oriented surface of genus $\gamma \geq 2$ and $g  \in \mathcal{M}_{-1}$. Consider a smooth map $u: (M,g) \to N$ with energy bounded by $E(u,g) \leq E_0$. Then there exists a constant $\eps_0 = \eps_0(N) > 0$ such that the following holds: Given $r>0$ such that
\begin{equation}
\label{eq:energyconcentration}
E(u,B_r(x)) \leq \eps_0
\end{equation}
for all $x \in M$, there exists $C< \infty$ (depending on $r$, $E_0$, $(M,g)$ and $N$) such that

\begin{equation}
\label{eq:h2bound}
\int_M   |\nabla u|^4 + |\nabla^2 u|^2  dv_g \leq C \left( 1  + \int  |\tau_g(u)|^2 dv_g \right).
\end{equation}

\end{lemma}
\begin{proof}
For $0< r < \inj(M,g)$ we have the following local regularity estimate (e.g. \cite{rupflin_existence}) for $u: M \to N$ with $E(u,B_r(x)) < \eps_0$:

\begin{equation}
\label{eq:localeps}
\int_{B_{\frac{r}{2}(x)}} |\nabla^2 u|^2 + |\nabla u|^4 dv_g \leq C\left( \frac{E(u,B_r(x))}{r^2} + \int_{B_r(x)} |\tau_g(u)|^2 dv_g \right).
\end{equation}

Now let $u$, $r$ be as in the lemma and further assume $r < \inj(M,g)$ for now. We can then cover $M$ by finitely many balls $B_{\frac{r}{2}}(x_i)$, $x_i \in M$, and apply estimate \eqref{eq:localeps} on each, using $E(u,B_r(x_i)) < \eps_0$. After summing and absorbing additional terms into the constant we obtain the claim.
Note that if $r \geq \inj(M,g)$, we can carry out the same argument replacing $r$ with e.g. $\frac{\inj(M,g)}{2}$.

\end{proof}

To apply this lemma, we need to find some $R>0$ such that $E(u_\eta(t),B_R(x)) < \eps_0$, i.e. control the concentration of energy along the flow. Assuming the existence of such an $R$ for a moment, we obtain our desired bound as a consequence of the following lemma.

\begin{lemma}
	\label{lem:W14bound}
	As before, take $N$ to be a smooth closed  manifold and $M$ a smooth oriented closed surface of genus $\gamma \geq 2$. Assume that there exists $R>0$ such that  $E(u_\eta(t),B_R(x)) < \eps_0$ for all $x \in M$, with $\eps_0$ as in Lemma \ref{lem:h2bound}.  Let $T>0, \eta_0 \geq \eta > 0$ and consider the solution $(u_\eta(t),g_\eta(t))$ with coupling constant $\eta$  to \eqref{eq:flow} defined on $[0,T]$ with initial data $(u_0,g_0)$. Then there exists $\eta_0 \geq \eta_1 > 0$ such that
	\begin{equation}
	\label{eq:claim1}
	\int_0^T \int_M |\nabla u_\eta|^4_{g_0} dv_{g_0} \leq C
	\end{equation}
	for all $\eta \leq \eta_1$ and some constant $C$, with $C$ and $\eta_1$ only depending on $(M,g_0)$, $u_0$, $R$, $N$ and $T$. Furthermore $|\nabla u_\eta(t)|_{g_0}$ is uniformly bounded in $L^2(M,g_0)$ for each $t \in [0,T]$ (with a bound only depending on $E_0=E(u_0,g_0)$).
\end{lemma}
\begin{proof}

We apply the bound \eqref{eq:h2bound} at each time $t \in [0,T]$ to the map $u_\eta(t)$ with respect to the initial metric $g=g_0$ and radius $r=R$ . We find
\begin{equation}
\label{eq:estL4}
\int_M |\nabla u_\eta|_{g_0}^4 dv_{g_0} + |\nabla^2 u_\eta|_{g_0}^2 dv_{g_0} dt \leq C(1 + \int |\tau_{g_0}(u_\eta)|^2) dv_{g_0},
\end{equation}
with a constant $C$ depending on $(M,g_0)$, $R$, $u_0$ and $N$, where we take $\nabla$ to denote the connection induced by $g_0$.

We would like to use the energy identity \eqref{eq:energy-identity} to integrate this inequality and further estimate the right hand side. Thus, as in \cite{rupflin_existence}, we want to replace $\tau_{g_0}(u_\eta)$ with $\tau_{g_\eta}(u_\eta)$ in \eqref{eq:estL4}. To do this, we estimate $|\tau_{g_{\eta}}(u_\eta)-\tau_{g_0}(u_\eta)|$ pointwise
\begin{equation}
\label{eq:tensionbound}
|\tau_{g_{\eta}}(u_\eta)-\tau_{g_0}(u_\eta)| \leq C \norm{g_\eta(t)-g_0}_{C^1(M,g_0)} ( |\nabla u_\eta|_{g_0}^2 + |\nabla u_\eta|_{g_0} + |\nabla^2 u_\eta|_{g_0}).
\end{equation}
This can be seen using  Remark \ref{rmk:eta} to in particular estimate the difference of the Christoffel symbols with respect to $g_0$ and $g_\eta(t)$, see also the author's thesis \cite{huxolThesis} for more detail. This allows us to estimate
\begin{align}
\begin{split}
\int |\tau_{g_0}(u_\eta)|^2 dv_{g_0} \leq C&\int |\tau_{g_\eta}(u_\eta)|^2dv_{g_0} \\
+ C& \norm{g_\eta(t)-g_0}_{C^1(M,g_0)} \int |\nabla u_\eta|_{g_0}^4 + |\nabla u_\eta|_{g_0}^2 + |\nabla^2 u_\eta|_{g_0}^2dv_{g_0},
\end{split}
\end{align}
which we can use together with estimate \eqref{eq:estL4} to find
\begin{align}
\begin{split}
\int_M |\nabla u_\eta|_{g_0}^4 dv_{g_0} + |\nabla^2 u_\eta|_{g_0}^2 dv_{g_0} dt \leq C& + C \int |\tau_{g_\eta}(u_\eta)|^2dv_{g_0}  \\
+C& \norm{g_\eta(t)-g_0}_{C^1(M,g_0)} \int |\nabla u_\eta|_{g_0}^4 + |\nabla^2 u_\eta|_{g_0}^2dv_{g_0},
\end{split}
\end{align}
where we used $\int |\nabla u_\eta|_{g_0}^2 dv_{g_0} \leq CE_0$ (as by Remark \ref{rmk:eta} the energy densities are comparable). Thus after choosing $\eta_1$ sufficiently small, now assuming $\eta \leq \eta_1$ we can absorb the remaining extra terms on the right and obtain
\begin{equation}
\label{eq:estL42}
\int_M |\nabla u_\eta|_{g_0}^4 dv_{g_0} + |\nabla^2 u_\eta|_{g_0}^2 dv_{g_0} dt \leq C ( 1+ \int |\tau_{g_\eta}(u_\eta)|^2)dv_{g_0}
\end{equation}
with a constant $C$ also only depending on $(M,g_0)$, $u_0$, $T$, $R$ and $N$. We now drop the $|\nabla^2 u_\eta|_{g_0}^2$ term (it was only required to control the error introduced by switching the metric of the tension) and integrate \eqref{eq:estL42} over $[0,T]$:
\begin{equation}
\int_0^T \int_M |\nabla u_\eta|_{g_0}^4 dv_{g_0} dt  \leq C ( 1+ E_0) \leq C\end{equation}
 Here we used the energy identity \eqref{eq:energy-identity} to estimate the integral of the tension.
 
Finally $|\nabla u_\eta(t)|_{g_0}$ is uniformly bounded in $L^2(M,g_0)$ for $t \in [0,T]$ by our assumption on the metric combined with the monotonicity of the energy $E(u_\eta(t),g_\eta(t))$.
\end{proof}

Hence it remains to show that the energy concentration along the flow is controlled.
We first recall that under certain assumptions on the metric (which in particular by the previous section hold in our case) the evolution of energy along the flow is controlled uniformly for short times.

\begin{lemma}[{Cf. \cite[Lemma 3.3]{rupflin_existence}, which in turn adapts \cite[Lemma 3.6]{struweCMH}.}]
	\label{lem:localenergybound}
	Assume $(u_\eta(t),g_\eta(t))$ to be a weak (as defined in \cite{rupflin_existence}) solution to \eqref{eq:flow} on $[0,T]$, with smooth initial data $(u_0,g_0)$ and some coupling constant $\eta > 0$. Further assume that there exists some $C_1 > 0$ such that for all $t \in [0,T]$ and smooth maps $u: M \to N$, we have $C_1^{-1} \norm{du}_{g_\eta(t)} \leq \norm{du}_{g_0} \leq C_1 \norm{du}_{g_\eta(t)}$, as well as $\norm{\ddt g_\eta(t)}_{C^0(M,g_0)} \leq C_1$. Then for all $r< \inj(M,g_0)$ the following estimate holds with a positive constant $\tilde{C}=\tilde{C}(C_1,g_0,u_0)$ for all $t \in [0,T]$, $0\leq \delta \leq T-t$:
	
	\begin{equation}
	E(u_\eta(t+\delta),B_{\frac{r}{2}}(x)) \leq \tilde{C} ( E(u_\eta(t), B_r(x)) +  \frac{\delta}{r^2}). 
	\end{equation}
	Here the energies and geodesic balls are taken with respect to the initial metric $g_0$.
	
\end{lemma}
\begin{proof}

	Take $\phi \in C^\infty_0(B_{r}(x), [0,1])$ to be a standard cut-off function, satisfying $ \phi \equiv 1$ on $B_\frac{r}{2}(x)$ and $|d \phi|_{g_0} \leq \frac{C}{r}$ (with some universal constant $C$). We can then multiply equation \eqref{eq:flow} for $u_\eta$ by $\phi^2 \pt u_\eta$ and integrate over $M$ with respect to $g_0$ (exactly as in \cite{rupflin_existence}) to arrive at
	\begin{equation}
	\label{eq:local0}
	0 = \int \phi^2 |\pt u_\eta|^2 dv_{g_0} - \int \phi^2 \pt u_\eta \lap_{g_{\eta}} u_\eta dv_{g_0}
	\end{equation}
	where we view the target as isometrically embedded via $N \hookrightarrow \mathbb{R}^n$. 
	Using integration by parts we find 
	\begin{align}
	- \int \phi^2 \pt u_\eta \lap_{g_{\eta}} u_\eta dv_{g_0} &=  \int \langle d (\phi^2 \pt u_\eta ), d  u_\eta \rangle_{g_\eta} dv_{g_0} \\
	&=  \int \pt u_\eta \langle d (\phi^2), d u_\eta \rangle_{g_\eta}dv_{g_0} + \int \phi^2 \langle d ( \pt u_\eta), d u_\eta \rangle_{g_\eta} dv_{g_0}. \label{eq:local1}	
	\end{align}
	We then note that 
	\begin{equation}
	\label{eq:local2}
	\frac{1}{2}	\ddt \int \phi^2 \langle  d u_\eta, d u_\eta \rangle_{g_\eta} dv_{g_0} = \int \phi^2 \langle d ( \pt u_\eta), d u_\eta \rangle_{g_\eta} dv_{g_0} + R(u_\eta,g_\eta)
	\end{equation}
	with an error term given by
	\begin{equation}
	R(u_\eta,g_\eta) = - \frac{1}{2} \int \phi^2 \langle \ddt g_\eta, du_\eta \tensor du_\eta \rangle_{g_\eta} dv_{g_0}
	\end{equation}
	which we can estimate using our assumptions as
	\begin{equation}
	\label{eq:inequalityR}
	|R(u_\eta,g_\eta)| \leq C \norm{ \ddt g_\eta}_{C^0(g_\eta)} E(u_\eta(t),g_\eta(t)) \leq C E(u_0,g_0) \leq C,
	\end{equation}
	with a constant $ C= C(C_1,u_0,g_0)$.
	We further estimate
	\begin{align}
	\begin{split}
	\label{eq:local3}
	\left| \int \pt u_\eta \langle d ( \phi^2), d u_\eta \rangle_{g_\eta}dv_{g_0}  \right| \leq& C \int |\phi| |\pt u_\eta| |d \phi|_{g_\eta} |d u_\eta|_{g_\eta} dv_{g_0} \\
	\leq& \int \phi^2 |\pt u_\eta|^2dv_{g_0} + C \int |d \phi|_{g_\eta}^2 |d u_\eta|^2_{g_\eta}dv_{g_0}.
	\end{split}
	\end{align}
	Note that we also have the bound $|d \phi|_{g_\eta(t)} \leq \frac{C}{r}$ with $C = C(C_1)$ by the assumption. Thus, combining \eqref{eq:local3}, \eqref{eq:local2}, \eqref{eq:local1} and \eqref{eq:inequalityR} with \eqref{eq:local0} we find
	\begin{equation}
	\label{eq:inequality1}
	\frac{1}{2}	\ddt \int \phi^2 | d u_\eta |^2_{g_\eta} dv_{g_0} \leq  \frac{C}{r^2} + C,
	\end{equation}
	where as before $C=C(C_1,u_0,g_0)$.
	Note that $r < {\inj}(M,g_0)$, hence $r$ is bounded from above in terms of only the genus of $M$, and we can simplify the above estimate to 
	\begin{equation}
	\label{eq:inequalityfinal}
	\frac{1}{2}	\ddt \int \phi^2 | d u_\eta |^2_{g_\eta} dv_{g_0} \leq  \frac{C}{r^2} 
	\end{equation}
	after adjusting the constant $C$ (now satisfying $C=C(M,g_0,u_0,C_1)$).
	
	We can integrate inequality \eqref{eq:inequalityfinal} over $[t,t+\delta]$ to find
	\begin{align}
	E(u_\eta(t+\delta), B_{\frac{r}{2}}(x)) &\leq C E(u_\eta(t+\delta),B_{\frac{r}{2}}(x); g_\eta(t+\delta) ) \leq C \int \phi^2 |d u_\eta(t+\delta)|_{g_\eta(t+\delta)}^2 dv_{g_0} \\ &\leq C (E(u_\eta(t),B_{r}(x);g_\eta(t)) + \frac{\delta}{r^2} ) \leq \tilde{C}( E(u_\eta(t),B_{r}(x)) + \frac{\delta}{r^2} )
	\end{align}
	which establishes the claim.
\end{proof}

This lemma allows to extend control of the energy concentration at a given time to nearby times. Thus, we can use Proposition \ref{prop:l2flowsmall} to establish $L^2$-closeness to a given smooth harmonic map flow.

\begin{lemma}
	\label{lem:shorttime}
	Take $N$ to be a smooth closed manifold and $M$ to be a smooth closed oriented surface of genus $\gamma \geq 2$ and $g_0 \in \mathcal{M}_{-1}$. Let $u_0 : M \to N$ be a smooth map. For $\eta>0$ denote by $(u_\eta(t),g_\eta(t))$  the solution to \eqref{eq:flow}. Then there exists a $0<\delta_0  = \delta_0(M,N,u_0,g_0)$ such that the conclusion of Proposition \ref{prop:l2flowsmall}  holds on the interval $[0,\delta_0]$: given $\eps>0$ we can find $\eta_\eps= \eta_\eps(M,N,u_0,g_0)>0$ such that for all $\eta < \eta_\eps$ we have
	\begin{equation}
		\norm{u(t)-u_\eta(t)}_{L^2(M,g_0)} < \eps.
	\end{equation}
\end{lemma}
\begin{proof}
 We may assume that $\eta \leq \eta_0$ so Remark \ref{rmk:eta} applies. Therefore by Lemma \ref{lem:localenergybound} we find, for all $x \in M$ and some $r_0$ to be chosen
	\begin{equation}
		E(u_\eta(\delta), B_{r_0}(x)) \leq \tilde{C}( E(u_0,B_{2r_0}(x)) + \frac{\delta}{2r_0^2}).
	\end{equation}
	The smoothness of $u_0$ implies that we can choose $r_0$ small enough such that $C E(u_0,B_{2r_0}(x)) < \frac{\eps_0}{2}$ for all $x \in M$, with $\eps_0$ as in Lemma \ref{lem:h2bound}. We then set $\delta_0 = \frac{r_0^2 \eps_0}{2\tilde{C}}$ to obtain $E(u_\eta(t),B_{r_0}(x)) < \eps_0$ for all $t \in [0, \delta_0]$. Thus $u_\eta(t)$ is smooth on $[0, \delta_0]$ (as no bubbles can develop, see \cite{rupflin_existence}), and therefore the assumptions of Proposition \ref{prop:l2flowsmall} are satisfied on $[0,\delta_0]$ by Lemma \ref{lem:W14bound}. 
\end{proof}

The important point is that the size of the interval $\delta_0$ only depends on the initial data. In the next section we will see how this allows us to set up an iteration argument (using the smoothness of the underlying harmonic map flow) to establish the closeness of $u$ and $u_\eta$ for small $\eta$ in our main theorem.

\subsection{Full Convergence of the Map}

So far we have established $L^2$-convergence of $u_\eta(t)$ to $u(t)$ for small times $t$. We first note that controlled concentration of energy together with standard parabolic regularity implies a priori bounds for $u_\eta(t)$: 
\begin{lemma}
	\label{lem:ckbounds}
		In the setting of Lemma \ref{lem:shorttime} we can choose $\tilde{\eta_\eps} \leq \eta_\eps$ such that additionally the  H\"older-norms (in space and time, and up to arbitrary order $C^{k}$) of $u(t)-u_\eta(t)$ on the interval $[0,\delta_0]$ stay bounded for $\eta \leq \tilde{\eta_\eps}$, with (uniform in $\eta$) bounds depending only on the initial data $(u_0,g_0)$, $T$, $k$, $M$ and $N$. In addition to the dependencies of $\eta_\eps$, $\tilde{\eta_\eps}$ may now also depend on $k$.
\end{lemma}
\begin{proof}
	We start with $\tilde{\eta_\eps} = \eta_\eps$. As before we see that $E(u_\eta(t),B_{r_0}(x)) < \eps_0$ for all $t \in [0, \delta_0]$. Hence the theory in \cite[Section 3]{rupflin_existence} applies as there is no concentration of energy up to time $\delta_0$. We now further choose $\tilde{\eta_\eps}$ small enough such that the estimate $\norm{g_0 - g_\eta(t)}_{H^s} \leq \epsilon_1$ holds, with $\epsilon_1= \epsilon_1(g_0,s)>0$ as defined in \cite[Equation (3.3)]{rupflin_existence}.
	The claim is then a direct consequence of \cite[Lemma 3.5, Remark 3.6, Remark 3.7]{rupflin_existence}, together with the smoothness of the harmonic map flow starting at $u_0$. 
\end{proof}
\begin{cor}
	\label{cor:higherordersmall}
	Under the assumptions of the previous lemma, for all $k \in \mathbb{N}$ we have $u_\eta(t) \to u(t)$ in $C^k$ as $\eta \downarrow 0$, uniformly for $t \in [0,\delta_0]$.
\end{cor}
\begin{proof}
	This follows by interpolation using the $C^k$-bounds provided by Lemma \ref{lem:ckbounds} and the $L^2$-convergence from Lemma \ref{lem:shorttime} (e.g. from Ehrling's Lemma).
\end{proof}

\begin{proof}[Proof of Theorem \ref{thm:eta0}]
	We already established the desired convergence of the metric, thus
	it remains to prove convergence of the map for all times.  
	By corollary \ref{cor:higherordersmall} the desired claim holds
	on $[0,\delta_0]$. In particular, we can choose
	$\eta_1=\eta_1(M,N,u_0,g_0)$ such that $\norm{u(t)-u_\eta(t)}_{W^{1,2}} <
	\frac{\eps_0}{4 \tilde{C}}$ for all $\eta \leq \eta_1$ and $t \in
	[0,\delta_0]$, where $\tilde{C}$ is the constant from Lemma
	\ref{lem:localenergybound}.
	
	As we assumed the harmonic map flow to be smooth up to time $T$,
	we can find $r_0$ such that $E(u(t),B_{2r_0}(x)) < \frac{\eps_0}{4
		\tilde{C}} $ for $t \in [0,T]$, $x \in M$. Hence,
	\begin{equation}
	E(u_\eta(\delta_0),B_{2r_0}(x)) \leq
	\norm{u(\delta_0)-u_\eta(\delta_0)}_{W^{1,2}} + \frac{\eps_0}{4\tilde{C}}
	\end{equation}
	and we find
	\begin{equation}
	E(u_\eta(\delta_0),B_{2r_0}(x)) <
	\frac{\eps_0}{2\tilde{C}}
	\end{equation}
	for all $\eta \leq \eta_1$.
	
	We can combine this estimate with Lemma \ref{lem:localenergybound}
	(as in the proof of Lemma \ref{lem:shorttime}) to find $\tilde{\delta_0} =
	\frac{\eps_0}{8\tilde{C}r_0^2}$ such that $E(u_\eta(t),B_{r_0}(x)) <
	\eps_0$ for $x \in M$, $t \in [\delta_0,\tilde{\delta_0}+\delta_0]$ . We
	can therefore extend Corollary \ref{cor:higherordersmall} up to time
	$\tilde{\delta_0}+\delta_0$. As the size of $\tilde{\delta_0}$ only depends
	on the initial data, we can iterate this process and obtain control on the energy density of $u_\eta(t)$ on the full time interval $[0,T]$, and hence also obtain Corollary \ref{cor:higherordersmall} on $[0,T]$.
\end{proof}
\section{A rescaled Limit}
\label{sec:rescaled}
We carry out the proof of Theorem \ref{thm:rescaled} in this section. To this end, we first establish estimates for the behaviour of the metric under the rescaled flow. We then compute the evolution of the $L^2$-norm of the tension, and as a consequence obtain that it becomes small for large $\kappa$. Leveraging this small tension we obtain uniform bounds on the rescaled flows as well as closeness to the (in our setting) unique harmonic map with respect to the metric at each time.

\subsection{Metric control on large time intervals for small $\eta$}

Analysing the flow \eqref{eq:flow2} on a time interval $[0,T]$ corresponds to studying the original flow \eqref{eq:flow} on $[0, \kappa T]$. In particular, for fixed $T$ we consider longer and longer time intervals as $\eta \to 0$, as opposed to the last section, thus local energy evolution is no longer controlled as in Lemma \ref{lem:localenergybound}.

Assuming that the injectivity radius along the flow stays controlled, techniques from \cite{RThorizontal} can be used to obtain uniform (in time) estimates on the metric.

By earlier metric estimates we find that at least for sufficiently small $T$ this injectivity radius control holds.

\begin{lemma}
	Let $M$ be a smooth oriented closed surface of genus $\gamma \geq 2$ and $N$ be a smooth closed Riemannian manifold. Assume that for some $T> 0$, $\kappa > 0$ we have a smooth solution $(u_\kappa(t),g_\kappa(t))$ to $\eqref{eq:flow2}$ on $[0,T]$ with rescaled coupling constant $\kappa$ for some given initial data $(u_0,g_0)$. Then there exists $T_0 = T_0(M,g_0,u_0) > 0$ (in particular \emph{independent} of $\kappa$) such that the injectivity radius $\inj_{g_\kappa}$ is  bounded away from 0 up to time $t=\min\{ T_0, T \}$.
	\label{lem:injectivity}
\end{lemma}

\begin{proof}

Analysing the corresponding unscaled flow $(u_\eta,g_\eta)$ on $[0,\kappa T]$, estimate \eqref{eq:L2small} yields 
\begin{equation}
	L(g_\eta,[0,\kappa T]) \leq \eta \sqrt{\kappa TE_0} = \sqrt{4T E_0}.
\end{equation}
We can thus choose $T_0$ by estimate \eqref{eq:Ckbound} 
such that $\inj_{g_\eta(t)} \geq \frac{1}{2} \inj_{g_0} > 0$ for $t \in [0,\kappa T_0]$ (cf. Remark \ref{rmk:eta}).
\end{proof}

\begin{rmk}
	The above claim could also be seen more geometrically by directly estimating the evolution of short closed geodesics along the flow using results in \cite{RTnonpositive}, see \cite{Hu} for such an approach.
	When the initial map is incompressible (\cite{RT}) the injectivity radius actually stays controlled uniformly in time, so restricting to small $T$ as above is not required. For targets with non-positive sectional curvature it is shown in \cite{RTnonpositive} that the injectivity radius cannot degenerate in finite time, however the relevant estimates unfortunately are not invariant under the rescaled time \eqref{eq:flow2}. 
\end{rmk}

Under such an injectivity radius bound for solutions of $\eqref{eq:flow}$, we can state the following $C^k$-estimate from \cite{RThorizontal} for horizontal curves (as defined in \ref{def:horizontal}).

\begin{lemma}[{Special case of \cite[Lemma 3.2]{RThorizontal}}]
	\label{lem:Ckmetric}
	Let $M$ be a smooth oriented closed surface of genus $\gamma \geq 2$. Let $\epsilon > 0 $ and consider  a horizontal curve of metrics  $g(t)$ on $M$ defined on the interval $[0,T]$, with uniformly bounded injectivity radius $\inj_{g(t)} \geq \epsilon$. Then there exists a $\delta > 0$, depending only on $\gamma$ and $\epsilon$, such that if the $L^2$-length (as defined in Lemma \ref{lem:l2small}) satisfies $L(g,[s,t]) < \delta$ for some $[s,t] \subset [0,T]$, then we have some $C_1 > 0$ only depending on $\gamma$ such that for any $t_1,t_2 \in [s,t]$ there holds
	
	\begin{equation}
	\label{eq:metricequiv1}
	g(t_1) \leq C_1 g(t_2).
	\end{equation}
	We further have some constant $C_2$ only depending on $M$, $\epsilon$ and $k$ such that 
	\begin{equation}
	|g(t_1) - g(t_2)|_{C^k(g(t_0))} (x) \leq C_2 L(g,[t_1,t_2]),
	\end{equation}
	for any $s \leq t_1 \leq t_2 \leq t$, $t_0 \in [s,t]$.
\end{lemma}

\begin{rmk}
	Note that as opposed to Proposition \ref{prop:metricbounds} cited in the last section, the constant $C_2$ here does not depend on the initial metric of the considered horizontal curve, but we do now additionally require uniform control of the injectivity radius along the whole curve $g(t)$.
\end{rmk}

We can now apply this lemma along horizontal curves $g(t)$ arising from solutions to \eqref{eq:flow2}. We first show that the metrics (assuming an injectivity radius bound) stay equivalent. 
\begin{cor}
	\label{cor:metricequiv}
	As usual, take $M$ to be a smooth closed oriented surface of genus $\gamma \geq 2$ and $N$ to be a smooth closed Riemannian manifold. Let $T > 0$, $\kappa > 0$ and consider a smooth solution $(u_\kappa(t),g_\kappa(t))$ to \eqref{eq:flow2} on the time interval $[0,T]$ with rescaled coupling constant $\kappa$ starting at the initial data $(u_0,g_0)$. Assume that there exists $\delta > 0$ such that $\inj_{g_\kappa(t)} \geq \delta >0$ for $t \in [0,T]$. Then we can find a constant $C > 0$, only depending on $\delta$, $M$, $E(u_0,g_0)$ and $T$ such that for any $s,t \in [0,T]$ we have
	\begin{equation}
	g_\kappa(s) \leq C g_\kappa(t).
	\end{equation}
\end{cor} 
\begin{proof}
	By Lemma \ref{lem:l2small}, we obtain a bound of the form
	\begin{equation}
	L(g_\kappa,[s,t]) \leq \sqrt{(t-s)E(u_0,g_0)}.
	\end{equation}
	Repeated application of Lemma \ref{lem:Ckmetric} on subintervals now yields the claim.

\end{proof}

Hence the metric along solutions of the flow \eqref{eq:flow2} stays uniformly equivalent to e.g. $g(0) = g_0$, as long as the injectivity radius is controlled, in particular up to the time $T_0$ from Lemma \ref{lem:injectivity}. We now observe that certain norms defined with respect to the changing metric also stay controlled. This problem was already considered in \cite[Section 3]{RThorizontal}, and the same methods directly apply in our situation.

Let $k$ be a non-negative integer here. Recall that we defined the $C^k(M,g)$-norm of (in particular) tensors $h \in \Sym^2(T^*M)$ via
\begin{equation}
\norm{h}_{C^k(M,g)} :=   \sup_{x \in M} \sum^k_{l=0} |\nabla^l h|_{g} (x),
\end{equation}
where $\nabla$ refers to the Levi-Civita connection on $(M,g)$ and its extensions.
Note that the same definition also extends to maps $u: M \to N \hookrightarrow \mathbb{R}^n$.

\begin{lemma}[Content from Section 3 of \cite{RThorizontal}] 
	\label{lem:ckequiv}
	Let $T>0, \kappa > 0$ and take $M$, $N$ and $(u_\kappa,g_\kappa)$ as in the above lemma. Assume again that there is some $\epsilon >0$ such that $\inj_{g_\kappa} \geq \epsilon$ on $[0,T]$. Then the $C^k(M,g_\kappa(t))$ norms as defined above are uniformly equivalent on $[0,T]$, in the sense that there exists some $C = C(k,\epsilon,M,E(u_0,g_0),T)> 0$ such that for any $t_1,t_2 \in [0,T]$ we have
	\begin{align*}
	\norm{h}_{C^k(M,g_\kappa(t_1))} \leq C  \norm{h}_{C^k(M,g_\kappa(t_2))}. 
	\end{align*}
	Note that $h$ here can be either a tensor or a map $h:M \to N \hookrightarrow \mathbb{R}^n$.
\end{lemma}
\begin{proof}
	This can be seen as in the proof of \cite[Lemma 3.2]{RThorizontal}, combined with standard estimates for $|\pt g_\kappa(t)|_{C^k(M,g(t))}$ (e.g. \cite[Lemma A.9]{RTZ}). 

\end{proof}

\subsection{Evolution of the tension}

A quantity that proved very important in the study of the harmonic map flow is the (squared) $L^2$-norm of the tension $$\Tau(u,g) = \norm{\tau_{g}(u)}^2_{L^2(M,g)} = \int_M |\tau_g(u)|^2.$$ For the classical harmonic map flow into nonpositively curved targets this turns out to be monotonically decreasing in time. This monotonicity can be seen by computing the second variation of the energy along solutions of the harmonic map flow, which is a well-known calculation (see e.g. \cite{ES}).
We show that for solutions to \eqref{eq:flow} (assuming the metric does not degenerate) with a target of nonpositive curvature we have a bound on how fast the tension can increase instead, which improves for small $\eta$. Note that the curvature hypothesis on $N$ means that any smooth initial data $(u_0,g_0)$ together with a choice of $\eta$ now leads to a smooth solution to \eqref{eq:flow} that exists for all times $t$ (see \cite{RTnonpositive}).

\begin{lemma}
	\label{lem:tauderiv}
	Assume $M$ as usual to be a smooth closed oriented surface of genus $\gamma \geq 2$ and $(N,G)$ to be a smooth closed Riemannian manifold, which we now also assume to have nonpositive curvature.
	Let $\eta > 0$ and take $(u(t),g(t))$ to be the (smooth) solution to \eqref{eq:flow} with coupling constant $\eta$, starting at initial data $(u_0,g_0)$. Assume that there exists some $\delta > 0$ such that $\inj_{g(t)} \geq \delta > 0$, and denote as usual $E_0 = E(u(0),g(0))$, then
	\begin{equation}
	\ddt \Tau(t) \leq CE_0^3 \delta^{-2} \eta^4
	\end{equation}
	where $\Tau(t) := \Tau(u(t),g(t))$ and $C < \infty$ only depends on the genus $\gamma$ of $M$.
\end{lemma}
\begin{proof}
	
	In the following we use the formalism of the induced covariant derivative $\nabla_t$ on $M \times [0,T]$, which agrees with $\pddt$ for time-dependent functions, as explained in e.g. \cite[p. 86ff]{Reto}. Following the same source, we also write $\nabla$ for the induced connection on the pullback bundles (i.e. $\nabla u$ is viewed as an element of $T^{*}M \tensor u^*(TN)$ etc.). Finally we consider $(N,G) \hookrightarrow \mathbb{R}^n$ to be isometrically embedded.
	We have
	$$\Tau(u,g) = \int_M g^{ij}\nabla_i \nabla_j u^k g^{mn} \nabla_m \nabla_n u^k dv_g.$$
	Hence
	\begin{equation}
	\label{eq:Tau3}
	\ddt \Tau = 2 \int_M \nabla_t ( g^{ij}\nabla_i \nabla_j u^k )g^{mn} \nabla_m \nabla_n u^k dv_g,
	\end{equation}
	where we used that the flow \eqref{eq:flow} leaves the induced volume form invariant (\cite{RT}), and thus no additional term involving the metric appears upon differentiating the integral.
	
	From now on we suppress the volume form as well as the superscript on the map, and write $h := \pddt g = \frac{\eta^2}{4}Re(P_g(\Phi(u,g)))$. Evaluating the first factor yields
	\begin{equation}
	\label{eq:Tau2}
	\nabla_t( g^{ij}\nabla_i \nabla_j u ) = - h^{ij} \nabla_i \nabla_j u  + g^{ij} \nabla_t \nabla_i \nabla_j u.
	\end{equation}
	We switch derivatives in the second term and obtain (for a derivation, see e.g. \cite[A.14, p. 86ff]{Reto})
	\begin{equation}
	g^{ij} \nabla_t \nabla_i \nabla_j u = g^{ij} \left( \nabla_i \nabla_j \pddt u + \Rm^{N}( \pddt u, \nabla_i u) \nabla_j u - (\pddt \Gamma^k_{ij}) \nabla_k u \right).
	\end{equation}
	From standard formulas we have the evolution of the Christoffel symbols given by
	\begin{equation}
	\pddt \Gamma^k_{ij} = \frac{1}{2} g^{kq} ( \nabla_i h_{jq} + \nabla_j h_{iq} - \nabla_q h_{ij} ) .
	\end{equation}
	Note that after tracing this with the metric $g$ (in $i,j$) it vanishes as $\delta h = \mathrm{tr} h = 0$ as holomorphic quadratic differentials are trace-free and divergence-free (\cite{HRT}).
	We now simplify \eqref{eq:Tau2} further, using the convention of repeated indices denoting traces (as we carried out the time derivatives now):
	\begin{equation}
	\nabla_t( g^{ij}\nabla_i \nabla_j u ) = - h^{ij} \nabla_i \nabla_j u +  \nabla_i \nabla_i \nabla_j \nabla_j u + \Rm^{N}( \nabla_k \nabla_k u, \nabla_i u) \nabla_i u.
	\end{equation}
	Putting this back into \eqref{eq:Tau3} we obtain
	\begin{equation}
	\ddt \Tau = 2 \int_M  \left(- h^{ij} \nabla_i \nabla_j u^k +  \nabla_i \nabla_i \nabla_j \nabla_j u^k + (\Rm^{N}( \nabla_p \nabla_p u, \nabla_i u) \nabla_i u )^k \right)\nabla_m \nabla_m u^k.
	\end{equation}
	We can view this integral as a sum of $L^2$-inner products (in the bundle $u^*(TN)$). Integrating the first (using $\delta h = 0$) and second term by parts yields
	\begin{multline}
	\ddt \Tau = 2 \int_M  h^{ij} \nabla_j u^k  \nabla_i \nabla_m \nabla_m u^k -  2 \int_M \nabla_i \nabla_j \nabla_j u^k \nabla_i \nabla_m  \nabla_m u^k + \\ 2 \int_M \langle \Rm^{N}( \tau_g(u), du(e_i)) du(e_i), \tau_g(u)) \rangle.
	\end{multline}	
	Using inner product notation we finally arrive at
	\begin{multline}
	\label{eq:Tau1}
	\ddt \Tau = 2 \int_M  \langle h_{ij} , \langle \nabla_i  u,  \nabla_j \tau_g(u) \rangle_{u^*(TN)} \rangle -  2 \int_M \langle \nabla \tau_g(u),  \nabla \tau_g(u) \rangle + \\ 2 \int_M \langle \Rm^{N}( du(e_i), \tau_g(u)) \tau_g(u), du(e_i) \rangle.
	\end{multline}	
	
	We now proceed to estimate the first term in \eqref{eq:Tau1}. Recall that given a lower bound $\delta$ on $\inj_g$ we can estimate
	\begin{equation}
	\norm{Re(\theta)}_{L^\infty(M,g)} \leq \norm{ \theta }_{L^\infty(M,g)} \leq C \delta^{-1} \norm{ \theta }_{L^1(M,g)}
	\end{equation}
	for any holomorphic quadratic differential $\theta$ with a constant $C < \infty$  depending only on $\gamma$ (\cite[Section 2]{RThorizontal}). We further know from \cite[Proposition 4.10]{RTnonpositive} that $P_g$ is a bounded operator from $L^1$ to $L^1$, i.e. 
	\begin{equation}
	\norm{ P_g(\phi) }_{L^1(M,g)} \leq C \norm{ \phi }_{L^1(M,g)}
	\end{equation}
	for any quadratic differential $\phi$ where $C < \infty$ again only depends on $\gamma$. Together with the uniform bound $\norm{ \Phi(u,g) }_{L^1(M,g)} \leq CE(u(t),g(t)) \leq CE_0$  we see that
	\begin{equation}
	\norm{Re(P_{g(t)}(\Phi(u(t),g(t))))}_{L^\infty(M,g)} \leq C \delta^{-1} \norm{ P_{g(t)}(\Phi(u(t),g(t)) }_{L^1(M,g)} \leq C \delta^{-1} E_0.
	\end{equation}
	Using this we estimate the first integral in \eqref{eq:Tau1} as
	\begin{align}
	&\left|\int_M \langle \frac{\eta^2}{4} Re(P_{g(t)}(\Phi(u(t),g(t)))_{ij} , \langle \nabla_i \tau_{g(t)}(u(t)), du(e_j)  \rangle \rangle \right| \\
	\leq& \int_M C \delta^{-1}  E_0 \frac{\eta^2}{4} |\langle \nabla \tau_{g(t)}(u(t)), du  \rangle| \\
	\leq& \int_M |\nabla \tau_{g(t)}|^2 + C \delta^{-2}  E_0^2 \eta^4 \int_M |du|^2 \\
	\leq& \int_M |\nabla \tau_{g(t)}|^2 + C \delta^{-2} E_0^3 \eta^4.
	\end{align}
	Here we used Young's inequality in the second inequality. From \eqref{eq:Tau1}, using the nonpositive sectional curvature of the target, we therefore obtain the claim
	\begin{equation}
	\label{eq:Tau4}
	\ddt \Tau(t) \leq CE_0^3 \delta^{-2} \eta^4. 
	\end{equation}
	
\end{proof}

We use this bound to see that for large $\kappa$ the flow \eqref{eq:flow2} very quickly has small tension. Note that the limiting $\eta = 0$ (or $\kappa = \infty$) case corresponds to the classical harmonic map flow, which satisfies $\Tau(t) \to 0$ as $t \to \infty$. Our estimate can be considered a quantitative version of this statement, allowing the metric to move slightly.
\begin{cor}
	\label{lem:odetau}
	With $M$, $N$, $(u(t),g(t))$ and $\delta$ as in the previous lemma, we have for any $\eps > 0$, and $\kappa = \frac{4}{\eta^2} \geq 1$
	\begin{equation}
	\Tau(t) \leq C(\eps) \kappa^{-1} \: \: , \: t \geq \eps \kappa,
	\end{equation}
	where $C(\eps) \to \infty$ for $\eps \to 0$ and $C(\eps)$ also depends on $E_0$, in addition to $\delta$ and the genus $\gamma$ of $M$.
	After carrying out the rescaling $\bar{t} = \frac{1}{\kappa}t$, this gives $\Tau(\bar{t}) \leq C(\eps)\kappa^{-1}$ for $\bar{t} > \eps$.
\end{cor}
\begin{proof}
	We note that Lemma \ref{lem:tauderiv} provides us with a linear estimate on $\Tau(t)$. Together with the $L^1$-bound $\int_0^T \Tau(t) \leq E(u_0,g_0)$ this implies the claimed point-wise bound, by simply comparing with an appropriate linear function and calculating the respective $L^1$-norm. 
	
	In particular, consider some time $t_0 \geq \eps \kappa$ and to simplify notation set $A :=  CE_0^3 \delta^{-2} \eta^4 = C \kappa^{-2}$ to be the derivative bound from the previous lemma and $h := \eps \kappa$. We want to show an upper bound for $\Tau(t_0)$. If $\Tau(t_0) \leq A h = C \eps \kappa^{-1}$, we take $Ah$ as our upper bound. Otherwise, define a (positive) linear function $f(t)$ by $f'(t) \equiv A$ and $f(t_0) = \Tau(t_0)$ on $[t_0-h,t_0]$. We find that $\ddt(\Tau(t)-f(t)) \leq 0$, hence $f(t) \leq \Tau(t)$ on $[t_0-h,t_0]$. Thus the $L^1$-norm of $f$ is bounded from above by the $L^1$-norm of $\Tau(t)$ (on $[t_0-h,t_0]$), and in particular by $E_0$. We compute the $L^1$-norm of $f$ as
	\begin{equation}
	\norm{f}_{L^1(t_0 - h,t_0)} = h( \Tau(t_0) - Ah) + \frac{1}{2} Ah^2 = h( \Tau(t_0) - \frac{1}{2}Ah).
	\end{equation}
	Therefore we have
	\begin{align}
	h( \Tau(t_0) - \frac{1}{2}Ah) \leq E_0 \\
	\Tau(t_0) \leq \frac{1}{h} E_0 + \frac{1}{2} Ah \\
	\Tau(t_0) \leq \kappa^{-1} (E_0 \eps^{-1}  + C \eps) = \kappa^{-1} C(\eps).
	\end{align}
	Thus $\Tau(t) \leq \kappa^{-1} C(\eps)$ in either case, proving the claim.
\end{proof}

\subsection{A priori estimates assuming small tension}
We now proceed to establish (uniform in time) a priori estimates for the flow \eqref{eq:flow} assuming an $L^2$-bound on the tension and a lower injectivity radius bound $\inj_g > r_0$ on the metric, aiming to then apply these estimates to the rescaled flow \eqref{eq:flow2}. 

\begin{rmk}
	\label{rmk:uniform}
Consider $(M,g)$ a closed hyperbolic surface, with injectivity radius bounded below by $r_0 < {\inj}_g$, for some $r_0 \leq 1$, and let $x \in M$. Let $r < r_0$, then we can choose particular local isothermal coordinates on the geodesic ball $B_r(x)$, which allow us to view $B_r(x)$ as a disk $(D_{r'},g_H)$ for some $r' < 1$ only depending on $r$, where $g_H$ denotes the Poincar\'{e} metric. This follows immediately by considering a local isometry from $(M,g)$ to the Poincar\'{e} disk, centred at $x$. We will also refer to these coordinates as \emph{hyperbolic isothermal coordinates}.
Hence we can write $g_H = \lambda^2 g_{eucl}$ where $\lambda : D_{r'} \to [2, K]$ denotes the conformal factor, with an upper bound given by a universal constant $K< \infty$ (as we considered disks of hyperbolic radius $r < r_0 \leq 1$, and thus the disk $D_{r'}$ stays away from the boundary of the Poincar\'{e} disk). 
Given a map $u:(M,g) \to N$, we can compute any  $W^{k,p}$-norm with respect to the metric $g$ on such a ball $B_r(x)$. Equivalently, we could work on $(D_{r'},g_H)$ as defined above. Note that instead we could also compute the $W^{k,p}$-norm of $u: (D_{r'},g_{eucl}) \to N$. This will give rise to an equivalent norm for the corresponding Sobolev topology as the conformal factor $\lambda$ stays controlled. In the remainder of this section, we will always view the arising norms as computed on such \emph{euclidean} disks, and also carry out the parabolic regularity theory on euclidean disks, unless specified otherwise.
\end{rmk}

For the classical harmonic map flow, assuming non-positive curvature on the target $N$ leads to uniform estimates for all times. This was first done in \cite{EL} by controlling the evolution of the energy density. 
In our setting, we begin by observing that controlled tension implies controlled energy density, using a bubbling argument exploiting the curvature hypothesis. This is similar to  \cite[Lemma 3.2]{RTnonpositive}.

\begin{lemma}
Let $(N,G)$ be a nonpositively curved smooth closed Riemannian manifold, and let further $r>0$ and take as usual $D_r$ to be the flat disk of radius $r$. Then for smooth maps $u: D_r \to N$ with energy bounded by $E(u, D_r) \leq E_0$, and $\eps >0$, there exists a constant $ 2 \leq K < \infty$, only depending on $N$, $E_0$ and $\eps$, such that for $x \in D_{\frac{r}{2}}$, we have $E(u,D_{r_\eps}(x)) < \eps$ whenever
\label{lem:energytension}
\begin{equation}
\label{eq:repsbound}
r_\epsilon < \frac{r}{ K ( 1 + r \norm{\tau(u)}_{L^2(D_r)}  )}. 
\end{equation}

\end{lemma}

\begin{proof}
Observe that we can restrict to the case $r=1$ and deduce the remaining cases by scaling. We argue by contradiction. Assume the lemma is false for some $\eps > 0$, $E_0 > 0$, then we obtain a sequence $r_i = \frac{1}{ i ( 1 + \norm{\tau(u_i)}_{L^2(D_1)} )} \in (0,\frac{1}{2})$ for $i \geq 2$ together with maps $u_i: D_1 \to N$ with $E(u_i, D_1) \leq E_0$ and points $x_i \in D_{\frac{1}{2}}$ with $E(u_i, D_{r_i}(x_i)) \geq \eps$. Note that $r_i \to 0$. We can therefore consider the restrictions of the maps $u_i$ to $D_{\frac{1}{2}}(x_i) \subset D_1$, and after shifting $x_i$ to the origin these form a sequence of maps (still labeled as $u_i$ for convenience) $u_i: D_{\frac{1}{2}} \to N$ with $E(u_i, D_{r_i}) \geq \eps$. We now rescale these maps by $\frac{1}{r_i}$ to obtain maps from larger and larger disks $\tilde{u}_i : D_{\frac{1}{2r_i}} \to N$, with $E(\tilde{u}_i, D_1) \geq \eps$. By the rescaling, the tension now satisfies
\begin{equation}
\norm{ \tau(\tilde{u}_i) }_{L^2} = r_i \norm{ \tau (u_i) }_{L^2(D_{\frac{1}{2}})} \leq \frac{1}{i} \frac{ \norm{\tau(u_i)}_{L^2(D_1)}} {1+  \norm{\tau(u_i)}_{L^2(D_1)} } \to 0.
\end{equation}
Therefore a standard bubbling argument, see e.g. \cite[Theorem 1.5]{HRT}, allows us to extract a nonconstant harmonic map $\tilde{u}_\infty : \mathbb{R}^2 \to N$, which can be extended to a (also nonconstant) harmonic map from $S^2 \to N$, which contradicts the curvature assumption on N (see e.g. \cite[Lemma 2.1]{RTnonpositive}). 
\end{proof}

We can apply the above local argument around each point on $M$ to obtain a version for maps $u:(M,g) \to N$.
\begin{cor}
	\label{lem:energyconctension}
	Let again $(N,G)$ be a nonpositively curved smooth closed Riemannian manifold and further take $M$ to be a smooth closed oriented surface of genus $\gamma \geq 2$. Given $r_0 > 0$ and $g \in  {\cal{M}}_{-1}$ with $\inj_{g_0} \geq r_0$, a smooth map $u: (M,g) \to N$ with bounded energy $E(u, g) \leq E_0$ and an $\epsilon >0$, there exists a constant $ K < \infty$, only depending on $\epsilon$, $r_0$, $N$ and $E_0$, such that for $x \in M$, we have $E(u,B_{r_\eps}(x)) < \eps$ whenever
	\begin{equation}	
	\label{eq:densitytension}
	r_\eps (u,g) < \frac{ 1 }{ K ( 1 + \norm{\tau(u)}_{L^2(M,g)}  )}.
	\end{equation}
	
\end{cor}

Thus we can obtain $W^{2,2}$-bounds on such maps of controlled energy and tension.

\begin{lemma}
	\label{lem:tension_apriori}
	Let $B_r(x)$ for some $x \in M$ be as above with radius $r < r_0 < {\inj}_{g}$ where $(M,g)$ is a smooth closed oriented hyperbolic surface as usual, $N$ has nonpositive sectional curvature, and $u:(M,g) \to N$ is a smooth map with energy bounded by $E(u,g) \leq E_0$ and tension bounded by $\norm{ \tau_{g}(u) }_{L^2(M,g)} \leq K$. Then we have $\norm{ \nabla^2 u }^2_{L^2(B_r(x) )} \leq C $, for some $C$ only depending on $r_0$, $E_0$, $K$ and $N$.
\end{lemma}
\begin{proof}
	We apply Lemma \ref{lem:energyconctension} with $\eps = \eps_0$ from Lemma \ref{lem:h2bound}, 
		giving us a radius $r_\eps$ such that in
		 particular $E(u,B_{r_\eps}(y)) < \eps_0$ for all $y \in M$. 
	Hence estimate \eqref{eq:h2bound} applies, and using $\norm{ \tau_{g}(u) }_{L^2(M,g)} \leq K$ we find
	\begin{equation}
	\int_M   |\nabla^2 u|^2  dv_g \leq C.
	\end{equation}
	Restricting to $B_r(x)$ we obtain the claim. Note that we understand  $\norm{ \nabla^2 u }^2_{L^2(B_r(x) )}$ in the sense of Remark \ref{rmk:uniform} (i.e. as computed on a flat disk).
\end{proof}

Using this $H^2$-bound we then apply standard parabolic theory to the equation for the map.

\begin{lemma}
	\label{lem:parabolic}
	Given  a smooth closed oriented surface $M$ of genus $\gamma \geq 2$ and $N$ a smooth closed Riemannian manifold with nonpositive sectional curvature, let $(u(t),g(t))$ be a solution to \eqref{eq:flow} on $[0,T]$ with coupling constant $\eta$, starting at initial data $(u_0,g_0)$, with energy bounded by $E(u_0,g_0) \leq E_0$ and such that there exists $r_0 > 0$
	with ${\inj}_{g(t)} \geq r_0$. Further assume that the tension field satisfies $\norm{\tau_{g(t)}(u(t))}_{L^2(M,g(t))} < K$ for all t in some time interval $[T_1,T_2] \subset [0,T]$ of length at least 2. Choose some time $t_0 \in [T_1 + 1, T_2-1]$, then on any hyperbolic isothermal chart $U = B_r(x)$ for $r < \frac{r_0}{4}$ defined with respect to $g(t_0)$, we obtain a bound in the parabolic Sobolev space $W^{2,1}_p(B_{\frac{r}{2}}(x) \times [t_0-\frac{1}{2},t_0 + 1])$ for all $p < \infty$ of the form
	\begin{equation}
	\norm{u(t)}_{W^{2,1}_p } \leq C
	\end{equation} 
	where $C$ depends only on $r_0$, $E_0$, $p$, $K$ and $N$, assuming $\eta \leq \eta_0$ with $\eta_0 > 0$ only depending on $E_0$, $r_0$ and $\gamma$.
\end{lemma}

\begin{proof}
Observe that $u$ satisfies the parabolic equation $u_t - \lap_g(u) = A_g(u)(\nabla u, \nabla u)$. By Lemma \ref{lem:tension_apriori} and the embedding $W^{2,2}(M,g(t_0)) \hookrightarrow W^{1,p}(M,g(t_0))$ for $p < \infty$ we find $u_t - \lap_g(u) \in L^p(M,g(t_0))$, and the corresponding local bounds on $B_{\frac{r}{2}}(x)$. For sufficiently small $\eta_0$, these estimates also extend to $t \in [t_0-1,t_0+1]$ by Lemma \ref{lem:Ckmetric}, hence standard interior parabolic regularity  gives the desired result (e.g. \cite[IV, Theorem 9.1]{Lady}).
\end{proof}

\begin{cor}
	\label{cor:bootstrap}
	Let $k$ be a nonnegative integer and $\alpha \in (0,1)$. Under the assumptions of the previous lemma, we have the spatial H\"older norms $C^{k, \alpha}$ of $u(t_0)$ bounded (on some slightly smaller ball $B_{r'}(x)$), with bounds only depending on $k$, $\alpha$, $K$, $r_0$, $E_0$, $r'$ and $N$.
\end{cor}
\begin{proof}
	This is a standard bootstrapping argument, see also the proof of \cite[Theorem 3.8]{rupflin_existence}. 
\end{proof}

\begin{rmk}
	Note that we could also get control on higher order time derivatives of the map by considering an explicit formula for the projection operator $P_g$, which would lead to bounding (higher order) time derivatives of the metric coefficients and allow us to carry out an iteration argument similar to the above. This was done in \cite{rupflin_existence}. However, as we only apply these estimates to the rescaled flow \eqref{eq:flow2} (where the time regularity would degenerate) we do not do this here.
\end{rmk}

We also observe that by elliptic regularity the restriction of any \emph{harmonic} map $u: (M,g(t)) \to N$ to a chart $B_r(x)$ enjoys good estimates, as we assumed $N$ to have nonpositive sectional curvature, see e.g. \cite{ES}. 

\begin{lemma}
	\label{lem:harmonic}
	Consider again a smooth closed oriented surface $M$ of genus $\gamma \geq 2$ and $N$ a smooth closed Riemannian manifold with nonpositive sectional curvature. Let $u:(M,g) \to N$ be a harmonic map with energy bounded by $E(u,g) \leq E_0$, and assume that $\inj_g > r_0 > 0$. Take $B_r(x)$ as before with $r < r_0$, then for any nonnegative integer $k$ and $\alpha \in (0,1)$ the H\"older-norms $C^{k, \alpha}$ of $u$ on $B_r(x)$ are controlled in terms of $k$, $\alpha$, $E_0$, $r_0$, $r$ and $N$.
\end{lemma}

So far in this section we have only obtained $\emph{local}$ bounds, on flat disks, for harmonic maps and solutions to \eqref{eq:flow2} with a priori bounded tension and injectivity radius. In fact, using estimates established earlier, we can actually obtain bounds with respect to the fixed norm $\norm{\cdot}_{C^k(M,g_0)}$ away from $t = 0$ for any solution to \eqref{eq:flow2}.

\begin{prop}
	\label{thm:Ckbounds}
	Take $M$ to be a smooth closed oriented surface of genus $\gamma \geq 2$ and $N$ to be a smooth closed Riemannian manifold with nonpositive sectional curvature. Given initial data $(u_0,g_0)$, let $0 < T \leq T_0$ with $T_0$ from Lemma \ref{lem:injectivity}, $\kappa \geq 1$ large enough so that $\eta \leq \eta_0$ with $\eta_0$ from Lemma \ref{lem:parabolic}, and consider the associated solution $(u_\kappa(t),g_\kappa(t))$ of \eqref{eq:flow2} defined on $[0,T]$ with rescaled coupling constant $\kappa$, starting at $(u_0,g_0)$. Let further $k$ be a nonnegative integer and $\frac{1}{\kappa} < \epsilon$, then for $0 < \epsilon \leq t \leq T$, we can obtain bounds on $\norm{u_\kappa(t)}_{C^k(M,g_0)}$, in terms of $k$, $\epsilon$, $E(u_0,g_0)$, ($M,g_0)$ and $N$. Additionally, if $\bar{u}(t)$ is a harmonic map with respect to $g_\kappa(t)$ that is homotopic to $u_0$, we can bound $\norm{\bar{u}(t)}_{C^k(M,g_0)}$ in terms of $k$, $E(u_0,g_0)$, ($M,g_0)$ and $N$.
\end{prop}
\begin{proof}
	
	We begin by considering $(u_\kappa(t),g_\kappa(t))$. Note that we are now working with the rescaled equations \eqref{eq:flow2}.
	Observe that we have a uniform (for $t \in [0,T]$) lower bound $r_0$ on the injectivity radius $\inj_{g_\kappa(t)}$, in terms of $E(u_0,g_0)$ and $(M,g_0)$ by Lemma \ref{lem:injectivity}. We also have a uniform upper bound on $\norm{\tau_{g_\kappa}(u_\kappa)}_{L^2(g_\kappa(t))}$ for $t \in [\epsilon,T]$ in terms of $\epsilon$, $E(u_0,g_0)$, $M$ and $r_0$ by Lemma \ref{lem:odetau}. 
	
	Hence for any fixed time $t \geq \epsilon$, we can apply Corollary \ref{cor:bootstrap} to obtain estimates on $\norm{u_\kappa(t)}_{C^k(D_{r'},g_{eucl})}$, for any flat disk $D_{r'}$ corresponding (by taking a hyperbolic isothermal coordinate chart, as in Remark \ref{rmk:uniform}) to a geodesic ball $B_r(x)$, defined with respect to $g_\kappa(t)$ for e.g. $r = \frac{r_0}{4}$. 
	Note that this directly implies a bound on $\norm{u_\kappa(t)}_{C^k(B_r(x))}$, computed with respect to the hyperbolic metric on $D_{r'}$, as these norms are equivalent
	\begin{equation}
	\norm{u_\kappa(t)}_{C^k(B_r(x))} \leq C \norm{u_\kappa(t)}_{C^k(D_{r'},g_{eucl})}, 
	\end{equation} 
	with a constant $C$ only depending on $k$ and $\gamma$ (as $r$ can be bounded in terms of $\gamma$) which can be seen by a direct calculation. We can then bound $\norm{u_\kappa}_{C^k(M,g_\kappa(t))}$, as we can estimate it by simply taking the supremum of $\norm{u_\kappa(t)}_{C^k(B_r(x))}$ over all $x \in M$, and we obtain
	\begin{equation}
	\norm{u_\kappa(t)}_{C^k(M,g_\kappa(t))} \leq C(k, \epsilon, E(u_0,g_0), N, (M,g_0)).
	\end{equation}
	But now the result follows by Lemma \ref{lem:ckequiv}, as the norms $\norm{\cdot}_{C^k(M,g_\kappa(t))}$ for $t \in [0,T]$ are uniformly equivalent to $\norm{\cdot}_{C^k(M,g_0)}$.
	
	The corresponding claim for harmonic maps can be seen in the same way, now using Lemma \ref{lem:harmonic} instead to obtain the local bounds. This uses that by \cite{hartman} the energy of $\bar{u}(t)$ is controlled by the energy of $u(t)$ (both with respect to $g_\kappa(t)$).
\end{proof}

\subsection{$C^k$-closeness to harmonic maps using small tension}

We now exploit the small tension to see that the flow \eqref{eq:flow2} becomes $C^0$-close to a harmonic map at each (positive) time for sufficiently large $\kappa$, under appropriate topological assumptions on the initial map $u_0$. Specifically, we will assume that the homotopy class of $u_0$ contains no constant maps nor maps to closed geodesics in the target as well as that $N$ has strictly negative sectional curvature, which ensures that for any metric on $M$ there exists a unique harmonic map homotopic to $u_0$ (see \cite{hartman}).

\begin{lemma}
	\label{lem:C0}
	Let $M$ be a smooth closed oriented surface of genus $\gamma \geq 2$ and $(N,G)$ be a smooth closed Riemannian manifold, which we now additionally assume to have strictly negative sectional curvature. Fix a homotopy class $H$ of maps $u: M \to N$ that contains no constant maps nor maps to closed geodesics in the target. 
	Consider any smooth map $u \in H$ and $g \in \mathcal{M}_{-1}$ with injectivity radius ${\inj}_g \geq r_0 > 0$ and energy $E(u,g) \leq E_0$. Then given any $\epsilon > 0$ there exists some $\delta = \delta(\epsilon,r_0,E_0,N,H)> 0$ such that $\norm{\tau_{g}(u)}_{L^2(M,g)} < \delta$ implies $|u-\bar{u}|_{C^0} < \epsilon$ where $\bar{u}$ denotes the unique harmonic map with respect to $g$ in $H$.
\end{lemma}

\begin{proof}
	Assume the claim was false, then there exists an $\epsilon > 0$ such that we can find a sequence $(u_i,g_i)$ satisfying the assumptions of the lemma with $\norm{\tau_{g_i}(u_i)}_{L^2(M,g_i)} \to 0$ and $|u_i-\bar{u_i}|_{C^0} \geq \epsilon$ for the (unique in $H$) $g_i$-harmonic map $\bar{u_i}$. 
	
	As $\inj_{g_i} \geq r_0 > 0$ we can apply Mumford compactness (as stated in Theorem \ref{thm:mumford1}) and obtain a subsequence $(u_i,g_i)$ together with diffeomorphisms $f_i : M \to M$ such that $f^*_i(g_i) \to h$ smoothly for some limit metric $h \in \mathcal{M}_{-1}$. We modify the maps $u_i$ with the same diffeomorphisms and denote $v_i = u_i \circ f_i$, $h_i = f^*_ig_i$. Then 
	$\norm{ \tau_{h_i}(v_i) }_{L^2(M,h_i)} \to 0$ and  $|v_i-\bar{v_i}|_{C^0} \geq \epsilon$, for any harmonic map (homotopic to $v_i$) $\bar{v_i}$ with respect to $h_i$. To see this, first note that this map is necessarily unique, as the homotopy class of $v_i$ also contains no constant maps nor maps to closed geodesics in the target. Hence it is given by $\bar{v_i} = \bar{u_i} \circ f_i$, and the $C^0$ distance considered is invariant under diffeomorphisms. 
	
	We can now see from the modified bubbling analysis carried out in \cite[Lemma 3.2]{RT} that there exists a harmonic map $\bar{v}:(M,h) \to N$ such that $v_i \to \bar{v}$ strongly in $W^{1,p}( M \setminus S)$ for all $p \in [0,\infty)$ where 
	\begin{equation}
	S := \{ x \in M: \text{for any neighbourhood $\Omega$ of $x$,} \, \limsup_{i \to \infty} E(u_i,g_i, \Omega) \geq \epsilon_0 \}.
	\end{equation} 
	In particular, the set $S$ of points where energy concentrates is empty in our case as a consequence of the curvature assumption on $N$. Hence $v_i \to \bar{v}$ in $C^0$ (as $W^{1,p} \hookrightarrow C^0$ for $p > 2$). It remains to see that $\bar{v_i}$ and $\bar{v}$ necessarily become close (in $C^0$) to deduce a contradiction to $|v_i-\bar{v_i}|_{C^0} \geq \epsilon$. 
	By the $C^0$-convergence of $v_i \to \bar{v}$ we see that $\bar{v}$ is homotopic to $v_i$ for all sufficiently large $i$. Even though the $f_i$ are not necessarily homotopic to the identity, it therefore follows that $\bar{v}$ is not a constant map nor maps to a closed geodesic in the target.  
	We finally obtain that $\bar{v_i} \to \bar{v}$ (in $C^0$) using the continuous dependence of the harmonic map on the metric in our setting (see \cite{EL}). In particular, this requires the strictly negative sectional curvature on $N$ in addition to the topological condition satisfied by $\bar{v}$.
	
\end{proof}

We finally obtain that the flow becomes close to a harmonic map at all (positive) times for large $\kappa$ in $C^k(M,g_0)$ by interpolation.
\begin{cor}
	\label{cor:ckclose}
	Take $M$ and $(N,G)$ to be as in the previous lemma. Assume the homotopy class of $u_0$ does not contain maps to closed geodesics in the target or constant maps. Given initial data $(u_0,g_0)$ let $0 <T \leq T_0$ with $T_0$ from Lemma \ref{lem:injectivity} and consider the sequence $(u_\kappa,g_\kappa)_{\kappa = 1}^{\infty}$ of solutions to \eqref{eq:flow2} with rescaled coupling constant $\kappa$, starting at $(u_0,g_0)$. Let $\epsilon > 0$ and $k$ be a nonnegative integer, then for any $t\geq \epsilon$ we have $\norm{u_\kappa(t)-\bar{u}_\kappa(t)}_{C^k(M,g_0)} \to 0$ as $\kappa \to \infty$, where $\bar{u}_\kappa(t)$ denotes the unique harmonic map with respect to $g_\kappa(t)$ in the homotopy class of $u_0$. Furthermore, this convergence is uniform on the interval $[\epsilon,T]$.
\end{cor}

\begin{proof}
	Note that we have bounds (independent of $\kappa$, as long as $\kappa$ is large enough so that $\eta \leq \eta_0$ with $\eta_0$ from Lemma \ref{lem:parabolic}) on both $u_\kappa(t)$ and $\bar{u}_\kappa(t)$ in $C^k(M,g_0)$ by Proposition \ref{thm:Ckbounds} applied on $[\epsilon,T]$. But we also have $|u_\kappa(t) - \bar{u}_\kappa(t)|_0 \to 0$ as $\kappa \to \infty$ by Lemma \ref{lem:C0} for $t \in [\epsilon,T]$ (using the tension bound from Corollary \ref{lem:odetau} and the injectivity radius bound from Lemma \ref{lem:injectivity}).
	Therefore the claim follows by interpolation, using e.g. Ehrling's lemma.
	
\end{proof}

\subsection{Constructing the limit flow}

\subsubsection{Uniqueness of the limit flow}
In the setting of Theorem \ref{thm:rescaled}, a homotopy class satisfying the assumptions of Claim \ref{claim:2} defines a map $\mathcal{H}: \mathcal{M}_{-1} \to C^{\infty}(M,N)$ sending hyperbolic metrics to their associated (unique) harmonic map in the chosen homotopy class.
Then Theorem \ref{thm:rescaled} proves convergence of the flows \eqref{eq:flow2} to a flow satisfying
\begin{equation}
\label{eq:flowharmonic}
\pt g(t) = Re(\Phi(\mathcal{H}(g(t)),g(t)).
\end{equation}
In the proof of Theorem \ref{thm:rescaled}, we make use of the following uniqueness statement for this flow. We give the proof at the end of this section.
 
\begin{prop}
	\label{prop:uniq}
	Take $M$ as usual to be a smooth hyperbolic surface, and $(N,G)$ to be a smooth closed Riemannian manifold of strictly negative curvature. Further take initial data $(u_0,g_0)$ with $g_0 \in \mathcal{M}_{-1}$ and $u_0: M \to N$ a smooth map that cannot be homotoped to a constant map or to a map into a closed geodesic on $N$. Let $T>0$ and consider two differentiable one-parameter families of metrics $g_i(t): [0,T] \to \mathcal{M}_{-1}, i \in {1,2}$ such that:
	\begin{itemize}
		\item There exists some $\delta > 0$ such that $\inj_{g_i(t)}(M) \geq \delta$ for all $t \in [0,T]$.
		\item The metrics $g_i(t)$ satisfy \eqref{eq:flowharmonic} with initial data $(u_0,g_0)$.	
	\end{itemize}
	Then $g_1(t) = g_2(t)$ for all $t \in [0,T]$.
\end{prop}

\subsubsection{Existence of the limit flow}
Having established uniqueness, we now use the $C^k$-bounds for the metric in Lemma \ref{lem:Ckmetric} to show existence of a limit flow for the flows \eqref{eq:flow2} as $\kappa \to \infty$. To this end, we consider time-dependent H\"older functions valued in some Banach space. 

\begin{defn}
	\label{def:timespace}
	Let $X$ be a Banach space, then we can define $C^{0}([0,T], X)$ to be the Banach space of bounded continuous functions valued in $X$ on $[0,T]$ equipped with the norm
	\begin{equation}
	\norm{f}_{C_X^{0}}  = \sup_{t \in [0,T]} \norm{f(t)}_X = \norm{f}_0.
	\end{equation} 
	
	Similarly for $\alpha \in (0,1)$ we denote by $C^{0,\alpha}([0,T], X)$ the Banach space of functions $f: [0,T] \to X$ which are H\"older continuous with exponent $\alpha$, with the canonical norm
	\begin{align}
	\norm{f}_{C_X^{0,\alpha}} =& \sup_{t_1,t_2 \in [0,T], t_1 \neq t_2} \frac{ \norm{f(t_1) - f(t_2)}_X } { |t_1-t_2|^{\alpha}} + \sup_{t \in [0,T]} \norm{f(t)}_X \\
	=& [f]_{\alpha} + \norm{f}_0.
	\end{align}
\end{defn}

In particular, if we let  $X = C^k(\Sym^2(T^*M),g_0)$ for any nonnegative integer $k$, we see that the metric $g(t)$ of a solution to \eqref{eq:flow2} lies in $C^{0,\frac{1}{2}}([0,T], X)$, as a consequence of Lemmas \ref{lem:Ckmetric}, \ref{lem:ckequiv} and \ref{lem:l2small}. We will need a compactness statement for these time-dependent H\"older spaces.

\begin{lemma}
	\label{lem:compact1}
	Assume $X$, $Y$ are Banach spaces such that $X$ compactly embeds into $Y$. Let $\alpha \in (0,1)$, then the embedding $C^{0,\alpha}([0,T], X) \hookrightarrow C^{0}([0,T], Y)$ is compact.
\end{lemma}

\begin{proof}
	This is a variant of the Arzela-Ascoli theorem. See e.g. \cite{huxolThesis} for a proof. 
\end{proof}

\begin{lemma}
	\label{lem:compact2}
	For $k$ a nonnegative integer the embedding $$C^{k+1}(\Sym^2(T^*M),g_0) \hookrightarrow C^{k}(\Sym^2(T^*M),g_0)$$ is compact.
\end{lemma}

\begin{proof}
	This is standard (and in fact works for arbitrary vector bundles, not just $\Sym^2(T^*M)$), and follows from the usual Arzela-Ascoli theorem, see e.g. \cite[Corollary 9.14]{AH_RF}.
\end{proof}

Using this compactness we can now finish the proof of Theorem \ref{thm:rescaled}.

\begin{proof}[Proof of Theorem \ref{thm:rescaled}]
	
	We will use the shorthand notation $C^k$ to refer to the space \\ $C^k(\Sym^2(T^*M),g_0)$.
	Let $k \in \N$, then each $g_\kappa$ is an element of $C^{0,\frac{1}{2}}([0,T], C^{k+1})$, as observed above. We denote the norm on this space by $\norm{\cdot}_{C_{k+1}^{0,\frac{1}{2}}}$. Also observe that we have a uniform (independent of $\kappa$) bound on $\norm{g_\kappa}_{C_{k+1}^{0,\frac{1}{2}}}$ by 
	Lemmas \ref{lem:Ckmetric}, \ref{lem:ckequiv} and \ref{lem:l2small} together with the bound on the injectivity radius from Lemma \ref{lem:injectivity}. Hence the $g_\kappa$ form a bounded sequence in $C^{0,\frac{1}{2}}([0,T], C^{k+1})$, and therefore we can find a convergent subsequence in $C^{0} ([0,T], C^k)$, which we again denote by $g_\kappa$, by Lemmas \ref{lem:compact1} and \ref{lem:compact2}, which converges to a limit $g$ in $C^{0} ([0,T], C^k)$. By repeating this subsequence argument we see that the limit $g$ lies in $C^{0} ([0,T], C^k)$ for all $k \in \N$, and we may assume that $g_\kappa$ converges in $C^{0} ([0,T], C^k)$ (again, for all $k \in \mathbb{N}$).
	
	Note that $g(t)$ is necessarily a metric for all $t$ by the uniform equivalence of the metrics $g_\kappa(t)$ from Corollary \ref{cor:metricequiv}, and we have $g(t) \in \mathcal{M}_{-1}$ as in particular the curvatures $R(g_\kappa(t))$ converge (by taking $k \geq 2$).
	This proves Claim \ref{claim:1}.
	
	To see the next claim, fix some $\epsilon > 0$ and some $k \in \mathbb{N}$.
	By Proposition \ref{prop:uniq} it suffices to prove Claim \ref{claim:2} up to a choice of subsequence $g_\kappa$, and by Claim \ref{claim:1} we can assume that our chosen subsequence converges to some $g$ as above.
	
	We then first show that  $\ddt g_\kappa$ converges to a limit in the space $C^{0} ([\epsilon,T], C^k)$. We will denote the norm on this space by $\norm{\cdot}_{C_k^0}$ (see Definition \ref{def:timespace}). Set $\Psi(t) = \Phi(u(t),g(t))$, where $u$ denotes the curve of harmonic maps associated to the limit curve of metrics $g$. We claim that $\ddt g_\kappa \to Re(\Psi)$ in $C^{0} ([\epsilon,T], C^k)$. Let $\Psi_\kappa(t) = \Phi(u_\kappa(t),g_\kappa(t))$, then we can estimate 
	
	\begin{align}
	\label{eq:proj1}
	\begin{split}
	\norm{\ddt g_\kappa - Re(\Psi)}_{C_k^0} =& \norm{Re(P_{g_\kappa}(\Psi_\kappa)) - Re(\Psi)}_{C_k^0} \\ \leq& \norm{Re(P_{g_\kappa}(\Psi_\kappa)) - Re(\Psi_\kappa)}_{C_k^0} 
	+ \norm{Re(\Psi_\kappa) - Re(\Psi)}_{C_k^0}.
	\end{split}
	\end{align}
	
	Note that indeed $\ddt g_\kappa = Re(P_{g_\kappa}(\Psi_\kappa))$ is an element of $C^{0} ([\epsilon,T], C^k)$ for each $\kappa$, as each $(u_\kappa(t),g_\kappa(t))$ is a smooth flow. Thus it is sufficient to show that the right hand side of \eqref{eq:proj1} converges to 0 as $\kappa \to \infty$ (in particular, that will also show that $Re(\Psi) \in C^{0} ([\epsilon,T], C^k)$).
	
	We start by estimating $\norm{Re(P_{g_\kappa}(\Psi_\kappa)) - Re(\Psi_\kappa)}_{C_k^0}$. This requires us to bound $$\norm{Re(P_{g_\kappa(t)}(\Psi_\kappa(t))) - Re(\Psi_\kappa(t))}_{C^k}$$ at each time $t \in [\epsilon,T]$ uniformly in $t$. We first bound the $L^1$-norm of this tensor using an elliptic Poincar\'{e} estimate for quadratic differentials from \cite{RT2}. This tells us
	\begin{align}
	\norm{Re(P_{g_\kappa(t)}(\Psi_\kappa(t))) - Re(\Psi_\kappa(t))}_{L^1(M,g_\kappa(t))} &\leq \norm{P_{g_\kappa(t)}(\Psi_\kappa(t)) - \Psi_\kappa(t)}_{L^1(M,g_\kappa(t))} \\
	&\leq C \norm{ \bar{\partial} \Psi_\kappa(t)}_{L^1(M,g_\kappa(t))},
	\end{align}
	where $C>0$ is some constant only depending on the genus $\gamma$ of $M$. By a standard calculation, see e.g. \cite[Lemma 3.1]{RT}, we can estimate 
	\begin{equation}
	\norm{ \bar{\partial} \Psi_\kappa(t)}_{L^1(M,g_\kappa(t))} \leq \sqrt{2} \norm{ \tau_{g_\kappa(t)}(u_\kappa(t)) }_{L^2(M,g_\kappa(t))} E(u_\kappa(t),g_\kappa(t))^{\frac{1}{2}}.
	\end{equation}
	From Corollary \ref{lem:odetau} we obtain that $\norm{ \tau_{g_\kappa(t)}(u_\kappa(t)) }_{L^2(M,g_\kappa(t))} \to 0$ uniformly in $t$ (for $t \in [\epsilon,T]$), and as usual we can bound $E(u_\kappa(t),g_\kappa(t)) \leq E(u_0,g_0)$, hence 
	\begin{equation}
	\norm{Re(P_{g_\kappa(t)}(\Psi_\kappa(t))) - Re(\Psi_\kappa(t))}_{L^1(M,g_\kappa(t))} \to 0
	\end{equation} uniformly in $t$ (again for $t \in [\epsilon,T]$). Hence convergence in $C^k$ follows by a standard interpolation argument if we can bound the $C^{k+1}$-norm.
	We observe that we can write the real part of the Hopf differential explicitly (see \cite{rupflin_existence}) as 
	\begin{equation}
	\label{eq:realhopf}
	Re(\Psi_\kappa(t)) = 2u_\kappa(t)^*G - 2e(u_\kappa(t),g_\kappa(t))g_\kappa,
	\end{equation} 
	which is therefore bounded in $C^{k+1}$, again uniformly in $t$, as a consequence of Proposition \ref{thm:Ckbounds}, together with the uniform boundedness of $g_\kappa(t)$ in $C^{k+1}$. 
	To bound $Re(P_{g_\kappa(t)}(\Psi_\kappa(t)))$, we can estimate
	\begin{equation}
	\norm{Re(P_{g_\kappa(t)}(\Psi_\kappa(t)))}_{C^{k+1}} \leq C \norm{P_{g_\kappa(t)}(\Psi_\kappa(t))}_{L^1(M,g_\kappa(t))} \leq C \norm{\Psi_\kappa(t)}_{L^1(M,g_\kappa(t))},
	\end{equation}
	with a constant $C$ only depending on a lower bound for $\inj_{g_\kappa(t)}$ and $\gamma$. This is a consequence of the fact that the $C^k$-norms of holomorphic functions are controlled by their $L^1$-norm (see \cite{RThorizontal}) combined with the $L^1-L^1$-boundedness of the projection operator $P_g$ (as already used in Lemma \ref{lem:tauderiv}, from \cite[Proposition 4.10]{RTnonpositive}). Thus $\norm{Re(P_{g_\kappa}(\Psi_\kappa)) - Re(\Psi_\kappa)}_{C_k^0} \to 0$ as $\kappa \to \infty$.
	
	We proceed to estimate the second term $\norm{Re(\Psi_\kappa) - Re(\Psi)}_{C_k^0}$. We again do this by bounding $\norm{Re(\Psi_\kappa(t)) - Re(\Psi(t))}_{C^k}$ uniformly in $t$ for $t \in [\epsilon,T]$. Take $\bar{u}_\kappa(t)$ as usual to be the unique $g_\kappa(t)$-harmonic map homotopic to $u_0$, and set $\bar{\Psi}_\kappa(t) = \Phi(\bar{u}_\kappa(t), g_\kappa(t))$. We find
	\begin{equation}
	\norm{Re(\Psi_\kappa(t)) - Re(\Psi(t))}_{C^k} \leq \norm{Re(\Psi_\kappa(t)) - Re(\bar{\Psi}_\kappa(t))}_{C^k} + \norm{Re(\bar{\Psi}_\kappa(t)) - Re(\Psi(t))}_{C^k}.
	\end{equation}
	As a consequence of Corollary \ref{cor:ckclose} we have $u_\kappa(t) \to \bar{u}_\kappa(t)$ in $C^k$ uniformly in $t$ for $t \in [\epsilon,T]$, and hence $\norm{Re(\Psi_\kappa(t)) - Re(\bar{\Psi}_\kappa(t))}_{C^k}$ converges to 0 uniformly in $t$ (again for $t \in [\epsilon,T]$) by the explicit formula \eqref{eq:realhopf}. 
	
	Finally, it remains to estimate $\norm{Re(\bar{\Psi}_\kappa(t)) - Re(\Psi(t))}_{C^k}$. We already know that $g_\kappa(t)$ converges to $g(t)$ in $C^k$ as $\kappa \to \infty$ uniformly in $t$ (by construction of $g$), so  by \eqref{eq:realhopf} we only need to check that the same holds for $\bar{u}_\kappa(t)$ and $u(t)$ (which will also prove the last statement of Claim \ref{claim:2}, as we already know $u_\kappa(t) \to \bar{u}_\kappa(t)$ in $C^k$). Note that the conditions to apply the results from \cite{EL}, are satisfied: $N$ has strictly negative sectional curvature and the homotopy class of $u_0$ does not contain maps to closed geodesics in the target or constant maps. Thus $\norm{\bar{u}_\kappa(t) - u(t)}_{C_k} \to 0$ can be deduced for each $t \in [\epsilon,T]$ by applying \cite[Theorem 3.1]{EL} on neighbourhoods of metrics covering the limit curve $g$. As $g(t) : [0,T] \to C^k(\Sym^2(T^*M))$ is a continuous function, a finite such cover can be found, which implies that the convergence is uniform in $t$ as claimed. 
	Therefore we also have $\norm{Re(\Psi_\kappa) - Re(\Psi)}_{C_k^0} \to 0$ as $\kappa \to \infty$.
	
	This establishes $\norm{\ddt g_\kappa(t) - Re(\Psi(t))}_{C_k^0} \to 0$ as $\kappa \to \infty$. As a consequence (e.g. using the fundamental theorem of calculus, or more abstractly viewing the sequence $g_\kappa$ as a convergent sequence in the Banach space $C^1([0,T],C^k)$, with limit necessarily equal to the curve $g$), we see that $g(t)$ is differentiable in $t$ on $[\epsilon,T]$, with derivative given by $\ddt g(t) = Re(\Psi(t))$.
	
	Thus Claim \ref{claim:2} is proved.

\end{proof}

\begin{rmk}
	One finds that the injectivity radius of a solution to the Teichm\"uller harmonic map flow \eqref{eq:flow} is bounded away from $0$ when the initial map is \emph{incompressible}, which we take to mean that $u_0 : M \to N$ is homotopically nontrivial and its action on the fundamental group of $M$ has trivial kernel (hence any simple closed homotopically nontrivial curve is mapped to another homotopically nontrivial curve in the target) (see \cite{RT}). 
	
	Thus the proof of the above theorem can be adapted (assuming $u_0$ is incompressible) to see that the limit flow exists for all time. We can then calculate the evolution of the energy along this limit flow $(u(t),g(t))$ and find
	\begin{equation}
	\ddt E(u(t),g(t)) = - \int_M  \frac{1}{4}|Re( \Phi(u,g))|^2  dv_g,
	\end{equation}
	allowing us to find a subsequence of times $t_i \to \infty$ with $\Phi(u(t_i),g(t_i)) \to 0$ (this is basically the same argument as in \cite{RT}). We can then continue arguing as in \cite{RT}, in particular applying Mumford compactness (see Theorem \ref{thm:mumford1}) to find a limit metric $\bar{g}$ and a limit weakly conformal  harmonic map $\bar{u}$, i.e. constant map or branched minimal immersions (after possibly adjusting by diffeomorphisms).
	
	Hence, assuming the target $N$ has strictly negative sectional curvature and the initial map $u_0$ is incompressible, and satisfies the homotopy class assumptions of Theorem \ref{thm:rescaled}, the rescaled flow \eqref{eq:flow2} converges to a limit flow which deforms the initial map into a branched minimal immersion (or constant map) through homotopic harmonic maps.
\end{rmk}

\subsubsection{Proof of Uniqueness}
We obtain the claimed uniquess in Proposition \ref{prop:uniq} as a consequence of estimates for changes of harmonic maps under deformations of the metric as found in \cite{EL}.
\begin{proof}[Proof of Proposition \ref{prop:uniq}]
	By \eqref{eq:flowharmonic}, the $g_i$ are horizontal curves. Let $u_i(t) = \mathcal{H}(g_i(t))$ for $t \in [0,T]$.	
	We can calculate the evolution of the energies $E(u_i,g_i)$
	\begin{equation}
	\ddt E(u_i,g_i) = - \int_M  \frac{1}{4}|Re( \Phi(u_i,g_i))|^2  dv_{g_i},
	\end{equation}
	as a consequence of the first variation formula for $E$ (e.g. \cite{RT}). Hence we have a bound on the $L^2$-length (as in Lemma \ref{lem:l2small}) and can therefore use Lemma \ref{lem:Ckmetric} to find 
	\begin{equation}
	\label{eq:ck1}
	\norm{g_i(t)-g_i(s)}_{C^k(g_0)} \leq C \sqrt{t-s}
	\end{equation}
	with a constant $C = C(M, \delta, k, E(u_0,g_0))$, as long as $s,t \in [0,T_0]$ with $T_0 = T_0 (M, \delta, E(u_0,g_0))$. We therefore have short-term uniform $C^k$ bounds on $g_i(t)$, and hence the set
	\begin{equation}
	I = \{t \in [0,T]: g_1(t) = g_2(t)\}
	\end{equation}
	is closed. It remains to prove $I$ is open. To this end, we may assume $T$ to be small.
	
	By the theory of \cite{EL} (in particular Theorem 3.1), there exists some $C=C(M,N,\alpha,k)$ such that
	\begin{equation}
	\norm{u_1(t) - u_2(s)}_{C^{k+1,\alpha}} \leq C \norm{g_1(t)-g_2(s)}_{C^{k,\alpha}} \label{eq:Ck2}
	\end{equation}
	for $s,t \in [0,T]$ with $T$ now chosen sufficiently small so that $g_1(t)$, $g_2(t)$ stay in the neighbourhood of metrics constructed in \cite{EL}. Here and in the following we take these H\"older norms to be computed with respect to the inital metric $g_0$.
	From \eqref{eq:Ck2} we see that the H\"older norms of $u_i(t)$ are uniformly bounded in terms of $\mathcal{H}(g_0)$, i.e. in terms of $M, g_0$ and $u_0$ (alternatively, we could obtain this by an argument as in Lemma \ref{lem:harmonic}). We now calculate the evolution of $\norm{g_1(t)-g_2(t)}_{C^{2,\alpha}}$: 
	
	\begin{equation}
	\ddt \norm{g_1(t)-g_2(t)}_{C^{2,\alpha}} \leq \norm{Re(\Phi(u_1(t),g_1(t)) - Re(\Phi(u_2(t),g_2(t)) }_{C^{2,\alpha}}.
	\end{equation}
	
	Using the formula \eqref{eq:realhopf} we can compute
	
	\begin{equation}
	Re(\Phi(u_1,g_1) - Re(\Phi(u_2,g_2) = 2u_1^*G - 2u_2^*G + 2e(u_1,g_1)g_1 - 2e(u_2,g_2)g_2.
	\end{equation}
	We can estimate
	\begin{equation}
	\norm{u_1^*G-u_2^*G}_{C^{2,\alpha}} \leq C \norm{u_1-u_2}_{C^{3,\alpha}} 
	\end{equation}
	with a constant $C$ only depending on higher $C^k$ bounds of $u_1$ and $u_2$, so $C=(M,g_0,u_0)$ as above.
	We next estimate $e(u_1,g_1)g_1 - e(u_2,g_2)g_2$:
	\begin{align}
	\norm{e(u_1,g_1)g_1 - e(u_2,g_2)g_2}_{C^{2,\alpha}} \leq
	\norm{e(u_1,g_1)g_1 - e(u_1,g_1)g_2}_{C^{2,\alpha}} \\
	+ \norm{e(u_1,g_1)g_2 - e(u_2,g_1)g_2}_{C^{2,\alpha}}\\
	+ \norm{e(u_2,g_1)g_2 - e(u_2,g_2)g_2}_{C^{2,\alpha}}.
	\end{align}
	We estimate each term on the right 
	\begin{align}
	\norm{e(u_1,g_1)g_1 - e(u_1,g_1)g_2}_{C^{2,\alpha}} \leq C\norm{g_1-g_2}_{C^{2,\alpha}} \\
	\norm{e(u_1,g_1)g_2 - e(u_2,g_1)g_2}_{C^{2,\alpha}} \leq C\norm{u_1-u_2}_{C^{3,\alpha}}  \\
	\norm{e(u_2,g_1)g_2 - e(u_2,g_2)g_2}_{C^{2,\alpha}}  \leq
	C\norm{g_1-g_2}_{C^{2,\alpha}},
	\end{align}
	where we again absorb $C^k$ norms of $u_1$,$u_2$ and now also $g_1,g_2$ into $C=C(M,g_0,u_0)$.
	
	Summarising, we obtain
	\begin{equation}
	\ddt \norm{g_1(t)-g_2(t)}_{C^{2,\alpha}} \leq C(\norm{g_1(t)-g_2(t)}_{C^{2,\alpha}} + \norm{u_1(t)-u_2(t)}_{C^{3,\alpha}} ) \leq C\norm{g_1(t)-g_2(t)}_{C^{2,\alpha}}
	\end{equation}
	where we used \eqref{eq:Ck2}. Hence for $t \in [0,T]$ we find
	\begin{equation}
	\norm{g_1(t)-g_2(t)}_{C^{2,\alpha}} \leq \norm{g_1(0)-g_2(0)}_{C^{2,\alpha}} e^{-Ct}.
	\end{equation}
	
	As $g_1(0)=g_2(0)$ we see that $g_1$ and $g_2$ coincide on $[0,T]$, and obtain the claimed uniqueness.
\end{proof}

\appendix
\section{Appendix}

We use the following compactness result for hyperbolic metrics $\{g_i\}$.

\begin{thm}[{Mumford compactness, e.g. \cite[Appendix C]{tromba}}]
	\label{thm:mumford1}
	Let $\epsilon > 0$ and $g_i \in \mathcal{M}_{-1}$ be hyperbolic metrics such that the lengths of their shortest closed geodesics $\ell(g_i)$ satisfy $\ell(g_i) \geq \epsilon$, then, after passing to a subsequence in $i$, there exists a sequence of orientation-preserving diffeomorphisms $f_i: M \to M$ and $\bar{g} \in \mathcal{M}_{-1}$ such that $f_i^*g_i \to \bar{g}$ smoothly.
\end{thm}

\bibliographystyle{plain}

\bibliography{bibTH}  
\end{document}